\documentclass[letterpaper,11pt]{article}
\usepackage[utf8]{inputenc}
\usepackage[english]{babel}

\usepackage{algorithm}

\usepackage{float}
\usepackage{algpseudocode}
\usepackage{placeins}

\usepackage{latexsym}
\usepackage{amssymb}
\usepackage{amsthm}
\usepackage{mathrsfs}
\usepackage{amstext}
\usepackage{graphicx} 
\usepackage{amsfonts}
\usepackage{wrapfig}
\usepackage{sidecap} 
\usepackage{latexsym}
\usepackage{pdfpages}
\usepackage{epic}
\usepackage{fullpage}
\usepackage{color}
\usepackage{curves}
\usepackage{subfigure}
\usepackage{setspace}
\usepackage{amsmath}
\numberwithin{equation}{section}
\usepackage{anysize}
\usepackage{rotating}
\usepackage{enumerate} 
\usepackage{hyperref}
\usepackage{bbm}

\usepackage{graphicx}
\usepackage{subfigure}

\newtheorem{theorem}{Theorem}[section] 
\newtheorem{definition}[theorem]{Definition}

\newtheorem{proposition}[theorem]{Proposition}
\newtheorem{lemma}[theorem]{Lemma}
\newtheorem{remark}[theorem]{Remark}
\def\R{{\mathbb R}}

\renewcommand{\leq}{\leqslant}
\renewcommand{\geq}{\geqslant}
\numberwithin{equation}{section}

\newcommand{\px}{\partial_{x}}

\newcommand{\pxx}{\partial_x ^2}
\newcommand{\pxxx}{\partial_x ^3}

\newcommand{\pt}{\partial_t}
\newcommand{\ptt}{\partial_t^2}

\newcommand{\pnu}{\partial_\nu}

\newcommand{\IOT}[1]{\int_{0}^T\int_\Omega #1 \,dx\,dt} 
\newcommand{\IGT}[1]{\int_{0}^T\int_{\Gamma_1} #1 \,dS\, dt}

\newcommand{\IOTT}[1]{\int_{-T}^T \int_{\Omega} #1 \,dx\,dt}
\newcommand{\IGTT}[1]{\int_{-T}^{T}\int_{\Gamma_1} #1 \,dS\, dt}

\definecolor{darkgreen}{rgb}{0.0, 0.5, 0.0} 

\title{Simultaneous reconstruction of two  potentials for a nonconservative Schr\"odinger equation with dynamic boundary conditions}
\author{
	Hugo Carrillo\thanks{School of Electrical Engineering, Pontificia Universidad Cat\'olica de Valpara\'iso, Valpara\'iso, Chile e-mail: hugo.carrillo@pucv.cl}
	\and
	Alberto Mercado\thanks{Departamento de Matem\'atica, Universidad T\'ecnica Federico Santa Mar\'{\i}a, Casilla 110-V, Valpara\'{\i}so, Chile e-mail: alberto.mercado@usm.cl}
	\and 
	Roberto Morales\thanks{IMUS, Universidad de Sevilla, Apartamento 1160, 41080 Sevilla, Spain  e-mail: rmorales1@us.es}
	}

\date{\today}
\begin{document}
\maketitle
\begin{abstract} 
	In this article, we consider an inverse problem involving the simultaneous reconstruction of two real valued potentials for a Sch\"odinger equation with mixed boundary conditions: a dynamic boundary condition of Wentzell type and a Dirichlet boundary condition. The main result of this paper is a Lipschitz stability estimate for such potentials from a single measurement of the flux. This result is deduced using the Bukhgeim-Klibanov method and a suitable Carleman estimate where the weight function depends on Minkowski's functional. 
\end{abstract}



\tableofcontents

\section{Introduction}
	Let $\Omega\subset \mathbb{R}^n$, $n\geq 2$, be a bounded domain with smooth boundary $ \partial \Omega$. 
    We assume that $\partial \Omega=\Gamma_0\cup \Gamma_1$ with $\Gamma_0$ and $\Gamma_1$ are two closed subsets satisfying $\Gamma_0\cap \Gamma_1=\emptyset$. 
    Unless explicitly stated otherwise, all of the function spaces discussed in this paper will concern complex-valued functions.
 
 We introduce the operators $L$ and $L_\Gamma$ given by
 \begin{align}
    \label{def:operator:L}
     L(y):=\,&i\pt y+d\Delta y -\vec{p_1}\cdot \nabla y\quad  \text{ in }\Omega\times (0,T)
\end{align}
and
\begin{align}
    \label{def:operator:L:gamma}
     L_\Gamma (y,y_\Gamma):=\,&i\pt y_\Gamma -d\pnu y+\delta \Delta_\Gamma y_\Gamma -\vec{p}_{\Gamma,1}\cdot \nabla_\Gamma y_\Gamma\quad  \text{ on }\Gamma_1\times (0,T),
 \end{align}
 where $i$ denotes the imaginary unit, $d,\delta>0$, $\vec{p}_1$ and $\vec{p}_{\Gamma,1}$ are vector valued functions defined in $\Omega$ and on $\Gamma_1$, respectively. Moreover,  $\pnu$ denotes the normal derivative associated to the outward normal $\nu$ of $\Omega$, $\nabla_\Gamma$ is the tangential gradient and $\Delta_\Gamma$ is the Laplace Beltrami operator. 
 
 In this article, we will consider the following nonconservative Schr\"odinger equation with dynamic boundary conditions:
	\begin{align}
	\label{original:problem:01}
	\begin{cases}
	L(y) + p(x)y=g&\text{ in }\Omega\times (0,T),\\
	L_\Gamma(y,y_\Gamma)+p_{\Gamma}(x)y_\Gamma =g_\Gamma&\text{ on }\Gamma_1\times (0,T),\\
	y\big|{_{\Gamma_1}} =y_\Gamma&\text{ on }\Gamma_1\times (0,T),\\
	y\big|_{\Gamma_0}=0&\text{ on }\Gamma_0\times (0,T),\\
	(y(\cdot,0),y_\Gamma(\cdot,0))=(y_0,y_{\Gamma,0})&\text{ in }\Omega\times \Gamma_1.
	\end{cases}
	\end{align} 
where $(g,g_\Gamma)$ is the source term, $p$ and $p_{\Gamma}$ are the potentials and $(y_0,y_{\Gamma,0})$ is the initial data.

We recall that, under the following assumptions on $(p,p_\Gamma) \in L^\infty(\Omega;\R)\times L^\infty(\Gamma_1;\R)$, $(y_0,y_{\Gamma,0})\in L^2(\Omega)\times L^2(\Gamma_1)$, $(g,g_\Gamma)\in L^1(0,T;L^2(\Omega)\times L^2(\Gamma_1))$ and suitable assumptions on $(\vec{p}_1,\vec{p}_{\Gamma,1})\in [L^\infty(\Omega)]^n \times [L^\infty(\Gamma_1)]^n$ (see Section \ref{section:preliminaries}) the problem \eqref{original:problem:01} has a unique weak solution. Moreover, there exists a positive constant $C=C(\Omega,T,p,p_\Gamma,\vec{p}_1,\vec{p}_{\Gamma,1})$ such that the solution $(y,y_\Gamma)$ of \eqref{original:problem:01} satisfies 
\begin{align*}
\|(y,y_\Gamma)\|_{C^0([0,T];L^2(\Omega)\times L^2(\Gamma_1))}\leq C\|(g,g_\Gamma)\|_{L^1(0,T;L^2(\Omega)\times L^2(\Gamma_1))}+C\|(y_0,y_{\Gamma,0})\|_{L^2(\Omega)\times L^2(\Gamma_1)}.
\end{align*}

In this paper, we are interested in an \textit{coefficient inverse  problem} associated to the Schr\"odinger equation with dynamic boundary conditions. More precisely, we consider the following:

{\bf Coefficient Inverse Problem (CIP):} Is it possible to retrieve $(p,p_{\Gamma}) \in L^\infty(\Omega;\R)\times L^\infty(\Gamma_1;\R) $ from a measurement of the normal derivative $\pnu y$ on $\Gamma_\star\times (0,T)$ ($\Gamma_\star\subseteq \Gamma_0$), where $(y,y_\Gamma)$ is the solution of \eqref{original:problem:01} associated to $(p,p_{\Gamma})$?
	
	We point out that our goal is the study of dependence of solutions $(y,y_\Gamma)$ of system \eqref{original:problem:01} with respect to the potentials $p$ and $p_{\Gamma}$. For this reason, we ignore for instance the dependence of the solutions of \eqref{original:problem:01} with respecto to the initial conditions and source terms. Besides, to emphasize this fact, sometimes we shall write 
    $$y= y[p,p_{\Gamma}] \quad \text{and} \quad  y_\Gamma = y_\Gamma[p,p_{\Gamma}].$$
    
	In this direction, the following  questions
     naturally arise:
	\begin{itemize}
		\item \textbf{Uniqueness:} Does the inequality $\pnu y[p,p_{\Gamma}]=\pnu y[q,q_{\Gamma}]$ on $\Gamma_\star\times (0,T)$ imply $p=q$ in $\Omega$ and $p_{\Gamma}=q_{\Gamma}$ on $\Gamma_1$?
		\item \textbf{Stability:} Is it possible to estimate $\|q-p\|_{L^2(\Omega)}$ and $\|q_{\Gamma} -p_{\Gamma}\|_{L^2(\Gamma_1)}$, by a suitable norm of $(\pnu y[q,q_{\Gamma}] - \pnu y[p,p_{\Gamma}])$ on $\Gamma_\star\times (0,T)$?
		\item \textbf{Reconstruction formula:} Can we find an algorithm to compute the potentials $p$ and $p_{\Gamma}$ by partial knowledge of $\pnu y[p,p_{\Gamma}]$ on $\Gamma_\star\times (0,T)$?
	\end{itemize}

In this article, we focus on the stability and reconstruction aspects of {\bf (CIP)}, properties derived under specific assumptions. Naturally, the stability result also ensures uniqueness.
 
We also consider the same inverse problem for a one-dimensional version of \eqref{original:problem:01}, which reads as follows.
\begin{align}
    \label{1D:original:problem}
    \begin{cases}
        i\pt y+d\pxx y-p_1(x)\px y+p(x)y=g(x,t) &\forall\, (x,t)\in (0,\ell)\times (0,T),\\
        i\dot{y}_{\Gamma}(t)-d \px y(0,t)+p_\Gamma y_{\Gamma}(t)=g_\Gamma(t) &\forall\, t\in (0,T),\\
        y(0,t)=y_{\Gamma}(t) &\forall\, t\in (0,T),\\
        y(\ell,t)=0 &\forall\, t\in (0,T),\\
        y(x,0)=y_0(x),\, y_\Gamma(0)=y_{\Gamma,0} &\forall\, x\in (0,\ell),
    \end{cases}
\end{align}
where $d>0$, $p_1\in L^\infty(\Omega)$, $(p,p_\Gamma)\in L^\infty(\Omega;\mathbb{R})\times \mathbb{R}$ and $(y_0,y_\Gamma)\in L^2(\Omega)\times \mathbb{R}$. Here, the pair $(y,y_{\Gamma_0})\in L^2(0,T;L^2(\Omega)\times \mathbb{R})$ stands for the state of the system \eqref{1D:original:problem}.

Our work is organized into two main parts. First, we derive a new Carleman estimate for the Schr\"odinger  operator with dynamic boundary conditions and use these estimates to apply the Bukhgeim-Klibanov method. 
Second, based on these results, we develop a constructive and iterative algorithm to determine the coefficients
$(p_0,p_{\Gamma,0})$ 
 using specific additional data. 
 This involves analyzing an appropriate functional derived from a data assimilation problem 
 and showing the existence of a unique minimizer in a suitable space, 
 from where we obtain the convergence of the iterative algorithm, which is  adapted from the Carleman-based approach introduced in
  \cite{Baudouin2013Global}.  
\subsection{Related references}
The first application of Carleman estimates  in the context of inverse problems is due to Bukhgeim and Klibanov in 1981 in \cite{bukhgeim1981global}. In this article, the authors proved H\"older stability results using a local Carleman estimate for complactly supported functions. After that, this method was slightly modified by Puel and Yamamoto in \cite{Puel1996On} by using a global Carleman estimate for the wave equation, allowing to obtain Lipschitz stability results for the Inverse source problem. See also \cite{Imanuvilov2001Global}, \cite{Imanuvilov2003Determination}, \cite{Klibanov2006Lipschitz} and \cite{Carreno2018Potential} for some related inverse problems for the wave equation and systems. 

The Carleman-based reconstruction algorithm (CbRec for short) was introduced in \cite{Baudouin2013Global} to study the reconstruction of the time independent potential of the wave equation posed in a bounded domain with Dirichlet boundary condition, with measurements on the flux of the solution in an appropriate subset of the boundary. This approach is inspired in the ideas given in \cite{Klibanov1995Uniform} and \cite{Klibanov1997Global} under additional assumptions. In \cite{Baudouin2013Global}, It is proved that this algorithm globally converges to the exact solution of the inverse problem for the wave equation, i.e., it converges to the unknown potential independently of the initial guest of the algorithm.

Unfortunately, the Carleman weights of the wave equation involves two exponential functions, which generates drawbacks in numerical simulations. In fact, in \cite{Cindea2013Numerical} the authors 
reconstructed numerically a control for the wave equation with a quadratic functional involving exponential weights with small values of the Carleman parameters, i.e., $s\approx 1$ and $\lambda \approx 0.1$. To avoid this problem, in \cite{Baudouin2017Convergent} the same authors 
studied the same problem, but this time a Carleman estimate was obtained by using a one parameter weight function. This allows to use the CbRec algorithm to reconstruction the unknown potential but the price to pay is additional observations are needed. These ideas have been studied to recover the speed coefficient of the wave in \cite{Baudouin2021Carleman} for the wave equation and in \cite{Boulakia2021Numerical} to reconstruct a spatial part of the source term in a reaction-diffusion equation. More recently, in \cite{baudouin2023Carleman} the CbRec algorithm has been implemented to address an inverse problem for the wave equation in a tree shaped network.

Concerning inverse problems for the Schr\"odinger equation, we mention the work \cite{Baudouin2002Uniqueness}, where the authors consider a coefficient inverse problem of recovering the zeroth order potential of the Schr\"odinger equation $i\pt v-\nabla\cdot(a(x)\nabla v)+p(x)v$ (with $a(x)\equiv 1$) with Dirichlet boundary conditions . Here, a Carleman estimate for the Schr\"odinger operator is proved using geometric conditions. Moreover, using the Bukhgeim-Klibanov method, a Lipschitz stability estimate is obtained. After that, in \cite{Baudouin2008AnInverse} consider a similar inverse problem but with $a(x)$ being a discontinuous continuous function. Once again, with a new Carleman estimate at hand, Lipschitz stability results have been obtained for such inverse problem. More recently, in \cite{Imanuvilov2023Sharp} an inverse source problem for general Schr\"odinger equation is studied. In that article, the authors prove stability results when no geometric conditions are assumed on the domain. The results relies on a different approach which depends on a transformation of the Schr\"odinger equation to an elliptic one. The key point here is such transformation is defined by a solution to a controllability problem for the one-dimensional Schr\"odinger equation. On the other hand, we refer to \cite{Lasiecka2004Globalp1}, \cite{Lasiecka2004Globalp2} and \cite{Triggiani2007Pointwise} on the study of well-posedness, stability and pointwise Carleman estimates for nonconservative Schr\"odinger equations with different boundary conditions. 

Only a few papers have been devoted to the study of inverse problems for PDEs with dynamic boundary conditions. In \cite{Chorfi2023Numerical}, the authors studies the identification of initial data for the heat equation with dynamic boundary condition. Recently, in \cite{chorfi2024lipschitz} the authors studied controllability issues and Inverse problem for the wave equation with kinetic boundary conditions. The findings of this article improve those obtained in \cite{Gal2017Carleman} with an additional geometric assumption. The key point used there is a new Carleman estimate for the wave operator with kinetic boundary conditions, where the associated weight function depends on the Minkowski functional, similar to those used in \cite{Mercado2023Exact}. More recently, in \cite{chorfi2024identification} the authors studied an inverse problem where the goal is to identity two spatial-temporal source terms for the Schr\"odinger equation with dynamic boundary conditions from final time measurements. To deal with it, the authors adopt a Tikhonov regularization strategy and analyze the properties of such functional. In particular, the existence and uniqueness of the solutions is investigated. In addition, some numerical experiments are given in one-dimensional case.

To the best of our knowledge, this is the first time that {\bf (CIP)} is considered for a non conservative Schr\"odinger equation with dynamic boundary conditions.

\subsection{Setting} 	
	In this section, we set up the notation and terminology used in this paper. Firstly, we consider the set $\Gamma_1\subset \partial \Omega$ as an $(n-1)$-dimensional compact Riemannian submanifold equipped by the Riemannian metric $g$ induced by the natural embedding $\Gamma_1 \subset \mathbb{R}^n$. We point out that it is possible to define the differential operators on $\Gamma_1$ in terms of the Riemannian metric $g$. However, for the purposes of this article, it will be enough to use the most important properties of the underlying operators and spaces. The details can be found, for instance, in \cite{Jost} and \cite{Taylor}. For the sake of completeness, we recall some of those properties.
	
	The tangential gradient $\Gamma_\Gamma$ of $y_\Gamma$ at each point $x\in \Gamma_1$ can be seen as the projection of the standard Euclidean gradient $\nabla y$ onto the tangent space of $\Gamma_1$, where $y_\Gamma$ is the trace of $y$ on $\Gamma_1$, i.e., 
	\begin{align*}
	\nabla_\Gamma y_\Gamma =\nabla y-\nu \pnu y,
	\end{align*}
	where $y=y_\Gamma$ on $\Gamma_1$ and $\pnu y$ is the normal derivative associated to the outward normal $\nu$. In this way, the tangential divergence $\text{div}_\Gamma$ in $\Gamma_1$ is given by
	\begin{align*}
	\text{div}_\Gamma (F_\Gamma):H^1(\Gamma_1)\to \mathbb{R},\quad y_\Gamma \mapsto -\int_{\Gamma_1} F_\Gamma \cdot \nabla_\Gamma y_\Gamma\, dS.
	\end{align*}
	Moreover, $\Delta_\Gamma$ denotes the Laplace-Beltrami operator, 
    which satisfies  $\Delta_\Gamma y_\Gamma=\text{div}_\Gamma (\nabla_\Gamma y_\Gamma)$ for all $y_\Gamma\in H^2(\Gamma_1)$. In particular, the surface divergence theorem holds:
	\begin{align}
    \label{surface:divergence:theorem}
	\int_{\Gamma_1} \Delta_\Gamma yz \,dS=-\int_{\Gamma_1} \nabla_\Gamma y \cdot \nabla_\Gamma z \,dS,\quad \forall y\in H^2(\Gamma_1),\quad \forall z\in H^1(\Gamma_1).
	\end{align}

    For $1\leq r\leq +\infty$, we define the Banach spaces $\mathbb{L}^r:=L^r (\Omega)\times L^r (\Gamma_1)$ endowed by their natural norms. In particular, the space $\mathbb{L}^2$ is a the Hilbert space endowed by the scalar product
	\begin{align*}
	\langle (u,u_\Gamma),(v,v_\Gamma) \rangle_{\mathbb{L}^2}:=\int_\Omega u\overline{v} \,dx + \int_{\Gamma_1} u_\Gamma \overline{v_\Gamma}\, dS.
	\end{align*}
	
	Moreover, for $m\in \mathbb{N}$, we consider the space
	\begin{align*}
	H_{\Gamma_0}^m(\Omega):=\{u\in H^m(\Omega)\,:\, u=0 \text{ on }\Gamma_0\},
	\end{align*}
	which is a closed subspace of the Sobolev space $H^m(\Omega)$. Moreover, for $m\in \mathbb{N}$, we set
 \begin{align*}
     \mathbb{H}_{\Gamma_0}^m:=\{(u,u_\Gamma)\in H_{\Gamma_0}^m(\Omega)\times H^m(\Gamma_1)\,:\, u=u_\Gamma\text{ on }\Gamma_1\}.
 \end{align*}
 
 We point out that, thanks to the Trace Theorem and Poincar\'e's inequality, $\mathbb{H}_{\Gamma_0}^1$ is a Hilbert space in $\mathbb{C}$ endowed by
	\begin{align*}
	\langle (u,u_\Gamma),(v,v_\Gamma) \rangle_{\mathbb{H}_{\Gamma_0}^1}:=\int_\Omega \nabla u \cdot \nabla \overline{v}\,dx +\int_{\Gamma_1} \nabla_\Gamma u_\Gamma \cdot \nabla_\Gamma \overline{v_\Gamma}\, dS, 
	\end{align*}
for all $(u,u_\Gamma),(v,v_\Gamma)\in \mathbb{H}_{\Gamma_0}^1$.	 
\subsection{Main results}
Now, we give a definition related with the geometric hypothesis we will assume for the interior boundary.
	
	\begin{definition}
		An open, bounded and convex set $U\subset \mathbb{R}^n$, is said to be \textbf{strongly convex} if $\partial U$ is of class $C^2$ and all the principal curvatures are strictly positive functions on $\partial U$.
	\end{definition}
\begin{remark}
    \label{remark:strongly:convex:set}
	We point out that a strongly convex set is geometrically strictly convex, in the sense that it has, at each of its boundary points, a supporting hyperplane with exactly one contact point.
\end{remark}
We will consider the following
\begin{itemize}
    \item[{\bf (A1)} ] {\bf (Geometric assumption)} We suppose that $\Omega$ can be written in the form $\Omega=\Omega_0\setminus \overline{\Omega_1}$, where $\Omega_1$ is strongly convex and
    $\Omega_0$ is an open set with $\overline{\Omega_1} \subset  {\Omega_0}$.
\end{itemize}

    Without loss of generality, we can  assume that $0\in \Omega_1$. Otherwise, take $x_0\in \Omega_1$ and apply a translation by $-x_0$. 
\begin{figure}[!h]
    \centering
    \includegraphics[width=0.4\linewidth]{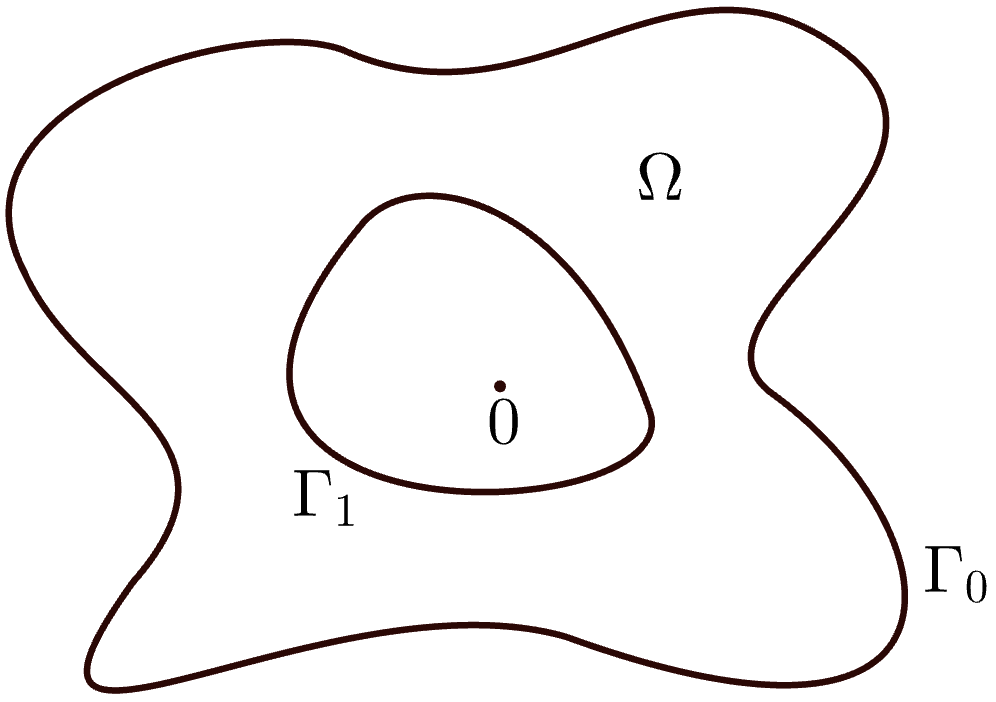}
    \caption{A domain $\Omega$ satisfying the geometric assumption {\bf (A1)}.}
    \label{fig:enter-label}
\end{figure}

From now on, we write $\Gamma_i:=\partial \Omega_i$, with $i=0,1$. We define, for each $x\in \mathbb{R}^n$,
	\begin{align}
	\label{def:psi}
	\mu(x)=\inf \{\lambda >0\,;\, x\in \lambda \Omega_1\},\quad\text{ and }\psi(x)=\mu^2(x).
	\end{align} 

Since $\Omega_1$ is open and convex, we have that 
$\mu$ and $\psi$ are well-defined and that 
$\psi=\mu=1$ on $\Gamma_1$.
In order to guarantee enough regularity for our purposes, we consider the following assumption:
\begin{itemize} 
\item[{\bf (A2)}] {\bf (Regularity of the boundary $\Gamma_1$)} We assume that the boundary $\partial \Omega_1$ has a regular parametrization, i.e., the mapping $x\in \partial \Omega_1 \mapsto \frac{x}{\|x\|}\in S^{n-1}$ is a $C^4$ diffeomorfism. 
\end{itemize}
Then we have (see Proposition 3.1 of \cite{Mercado2023Exact}) that $\mu\in C^4(\overline{\Omega})$,
$\nabla \mu \neq 0$ in $\overline{\Omega}$, and there exists $c>0$ such that 
    \begin{align*}
        D^2\mu(\xi,\xi)\geq c|\xi|^2\quad \text{in} \,\,  \overline{\Omega}, \quad \forall\,  \xi \in \mathbb{R}^n.
    \end{align*}

Define $\Gamma_\star\subseteq \Gamma_0$, which is our \textit{observation region}, as
\begin{align}
\label{def:Gamma:ast}
\Gamma_\star:=\{x\in \partial \Omega\,:\, \pnu \psi(x)\geq 0\}\subseteq \Gamma_0.
\end{align}

Now, for $m>0$ and $X\subset \mathbb{R}^n$, we introduce
 \begin{align*}
 L_{\leq m}^\infty(X;\R)=\{p\in L^\infty(X;\R)\,:\, \|p\|_{L^\infty(X)}\leq m \},
 \end{align*}
and define the set of potentials as 
\begin{align}
\label{def:L:infty:m}
\mathbb{L}_{\leq m}^\infty:= L_{\leq m}^\infty(\Omega;\R)\times L_{\leq m}^\infty(\Gamma_1;\R).
\end{align}

Now, we state the main result of this article, concerning {\bf (CIP)}.
 
 \begin{theorem}
 	\label{Thm:stability:boundary}
 	Consider the assumptions {\bf (A1)} and {\bf (A2)}. Define the function $\psi$ given in \eqref{def:psi}. Moreover, for given $m>0$ we consider $(q,q_\Gamma)\in \mathbb{L}_{\leq m}^\infty$. Assume that $\delta$ and $d$ satisfy the condition 
    \begin{align}
        \label{condition:d:delta}
        \delta>d.
    \end{align}
    
    In addition, suppose that
    
\begin{align} 
\label{regularity:stability}
(y[q,q_\Gamma],y_\Gamma[q,q_\Gamma])\in H^1(0,T;\mathbb{L}^\infty) 
\end{align}
 and the initial data $y_0$ and $y_{0,\Gamma}$ are real valued functions satisfying 
 	\begin{align}
\label{condition:initial:datum:stability}
 	|y_0|\geq r_0\text{ a.e. in }\Omega\quad \text{ and } \quad |y_{0,\Gamma}|\geq r_0\text{ a.e. on } \Gamma_1,
 	\end{align}
for some positive constant $r_0$. Then,
 	there exists a constant $C=(\Omega,T,p,p_\Gamma,y_0,y_{\Gamma,0},m)>0$ such that if
 	$(\pnu y[q,q_\Gamma]-\pnu y[p,p_\Gamma])\in H^1(0,T;L^2(\Gamma_*))$, then the following inequality holds
 	\begin{align}
 	\label{Lipschitz:inequality:Sch}
 	C^{-1}\|(q,q_\Gamma)-(p,p_\Gamma)\|_{\mathbb{L}^2} \leq \|\pnu y[q,q_\Gamma]-\pnu y[p,p_\Gamma] \|_{H^1(0,T;L^2(\Gamma_*))}\leq C\|(q,q_\Gamma)-(p,p_\Gamma)\|_{\mathbb{L}^2} 
 	\end{align}
    for all $ (p,p_\Gamma)\in \mathbb{L}_{\leq m}^\infty$.
 \end{theorem}



We point out that inequality  \eqref{Lipschitz:inequality:Sch} establishes a Lipschitz stability result for the inverse problem {\bf (CIP)} for the Schr\"odinger equation with dynamic boundary conditions from (partial) boundary observations given in $\Gamma_\star\subset \partial \Omega$.
\begin{remark}
Before going further, some remarks are in order:
\begin{itemize}
    \item The assumptions {\bf (A1)} and {\bf (A2)} are imposed in order to use one of the main ingredients in the proof of the Theorem \ref{Thm:stability:boundary}: a suitable Carleman estimate for the Schr\"odinger operator with dynamic boundary conditions (see \cite{Mercado2023Exact}).
    \item The condition \ref{condition:d:delta} also appears in the deduction of observability inequalities for the wave equation with acoustic boundary conditions where the observation is only in a portion of the domain where Dirichlet boundary condition is imposed. We refer to the reader to \cite{Baudouin2022A} and \cite{chorfi2024lipschitz}. 
    \item The positivity assumption on the initial data \eqref{condition:initial:datum:stability} required in order to apply the Bukhgeim-Klibanov method and Carleman estimates for inverse problems with only one boundary measurement; see for instance \cite{Yamamoto1999Uniqueness}, \cite{Baudouin2002Uniqueness} and \cite{Imanuvilov2023Sharp}.
\end{itemize}      
\end{remark}

\begin{remark}
Hypothesis \eqref{regularity:stability}
 is a typical technical condition for 
 obtaining stability in a one-measurement inverse problem. It can be satisfied by ensuring the system has sufficiently regular data.
\end{remark}

We can also formulate a stability result for the one-dimensional model \eqref{1D:original:problem}. To do this, for $1\leq r\leq \infty$, we define the spaces
\begin{align*}
    \mathcal{X}^r:=L^r(\Omega;\mathbb{R})\times \mathbb{R},
\end{align*}
endowed by its natural norm and for $m>0$, we also set the subspace 
\begin{align*}
    \mathcal{X}_{\leq m}^r:=\{(w,w_\Gamma)\in \mathcal{X}^m\,:\, \|(w,w_\Gamma)\|_{\mathcal{X}^r}\leq m \}.
\end{align*}

Then, we have the following result:

\begin{theorem}
    Given $m>0$, set $(p,p_\Gamma), (q,q_\Gamma)\in \mathcal{X}_{\leq m}^\infty$. Suppose that $(q,q_\Gamma)$, $p_1$, $(g,g_\Gamma)$ and $(y^0,y_\Gamma^0)$ are chosen such that 
    \begin{align*}
        (y[q,q_\Gamma],y_\Gamma[q,q_\Gamma])\in H^1(0,T;\mathbb{L}^\infty),
    \end{align*}
    where $y^0$ and $y_\Gamma^0$ are real valued functions which also satisfy
    \begin{align*}
        |y^0|\geq r_0\text{ a.e.  in }\Omega\text{ and }|y^0_\Gamma|>r_0,
    \end{align*}
    for some constant $r_0>0$. Then, there exists a constant $C=C(\Omega,T,p,p_\Gamma,y^0,y_{\Gamma}^0,m)>0$ such that if $(\px y[q,q_\Gamma](\ell,\cdot)-y[p,p_\Gamma](\ell,\cdot))\in H^1(0,T)$, then the following inequality holds 
    \begin{align*}
        C^{-1}\|(q,q_\Gamma)-(p,p_\Gamma)\|_{\mathcal{X}^2}\leq \|\px y[p,p_\Gamma](\ell,\cdot)-\px y[q,q_\Gamma](\ell,\cdot)\|_{H^1(0,T)}\leq C\|(q,q_\Gamma)-(p,p_\Gamma)\|_{\mathcal{X}^2},
    \end{align*}
    for all $(q,q_\Gamma)\in \mathcal{X}_{\leq m}^\infty$.
\end{theorem}

\subsection{Outline of the paper}
The rest of the article is organized as follows. In Section \ref{section:preliminaries} 
we review some basic results on the existence and uniqueness of solutions to the Schrödinger equation with dynamic boundary conditions. We also present a Carleman estimate for these systems with observations on part of the boundary. In Section \ref{section:proof:main:thm}, we prove the stability of the inverse problem discussed in Theorem \ref{Thm:stability:boundary}. Section \ref{section:convergence:CbRec} focuses on the convergence of the reconstruction algorithm. Finally, Section \ref{section:further:comments} provides concluding remarks and additional observations.


\section{Preliminaries}
\label{section:preliminaries}
In this section, we present existence, uniqueness, and other basic properties of the Schrödinger equation with dynamic boundary conditions. Additionally, we present  a Carleman estimate for this system, considering observations on a subset of the boundary.

\subsection{Existence and uniqueness of solutions and hidden regularity}
Given $d,\delta>0$, we consider the problem 
\begin{align}
    \label{wellposedness:schrodinger}
    \begin{cases}
        i\pt v +d\Delta v +\vec{\rho}_1\cdot \nabla v +\rho_0 v =h&\text{ in }\Omega\times (0,T),\\
        i\pt v_\Gamma -d\pnu v +\delta \Delta_\Gamma v_\Gamma +\vec{\rho}_{\Gamma,1} \cdot \nabla_\Gamma v_\Gamma +\rho_{\Gamma,0}v_\Gamma= h_\Gamma&\text{ on }\Gamma_1 \times (0,T),\\
        v\big|_{\Gamma_1}=v_\Gamma &\text{ on }\Gamma_1\times (0,T),\\
        v\big|_{\Gamma_0}=0&\text{ on }\Gamma_0\times (0,T),\\
        (v(\cdot,0),v_\Gamma(\cdot,0))=(v_0,v_{\Gamma,0})&\text{ in }\Omega\times \Gamma_1. 
    \end{cases}
\end{align}

Concerning the existence and uniqueness of \eqref{wellposedness:schrodinger}, we have the following result:

\begin{proposition}[Proposition 2.1 \cite{Mercado2023Exact}]\label{proposition:existence:01}
    Suppose that $(v_0,v_{\Gamma,0})\in \mathbb{L}^2$, $\vec{\rho}_1\in [W^{1,\infty}(\Omega)]^{n}$. Furthermore, suppose that $\vec{\rho}_{\Gamma,1}\in [W^{1,\infty}(\Gamma_1)]^n$ such that $\vec{\rho}_1=\vec{\rho}_{\Gamma,1}$ on $\Gamma_1\times (0,T)$, $(\rho_0,\rho_{\Gamma,1})\in \mathbb{L}^\infty$ and $(h,h_\Gamma)\in L^1(0,T;\mathbb{L}^2)$. Then, the weak solution $(v,v_\Gamma)$ belongs to $C^0 ([0,T];\mathbb{L}^2)$. Besides, there exists $C>0$ such that 
    \begin{align*}
        \max_{t\in [0,T]} \|(v,v_\Gamma)\|_{\mathbb{L}^2} \leq C\left( \|(h,h_\Gamma)\|_{L^1(0,T;\mathbb{L}^2)} + \|(v_0,v_{\Gamma,0})\|_{\mathbb{L}^2} \right). 
    \end{align*}
\end{proposition}
 
The following results relies on the existence and uniqueness of weak solutions of \eqref{wellposedness:schrodinger} for more regular data.

\begin{proposition}[Proposition 2.2 \cite{Mercado2023Exact}]\label{proposition:existence:02}
    Suppose that $(v_0,v_{\Gamma,0})\in \mathbb{H}_{\Gamma_0}^1$ and $(h,h_\Gamma)\in L^1(0,T;\mathbb{H}_{\Gamma_0}^1)$. In addition, assume that $(\vec{\rho}_1,\vec{\rho}_{\Gamma,1})\in [W^{1,\infty}(\Omega)]^n \times [W^{1,\infty}(\Gamma_1)]^n$ satisfies the following conditions 
    \begin{align*}
        \vec{\rho}_1=\vec{\rho}_{\Gamma,1}\text{ on }\Gamma_1\times (0,T)\text{ and }\vec{\rho}_1\cdot \nu \leq 0\text{ on }\partial \Omega\times (0,T).
    \end{align*}

    We also assume that $(\rho_0,\rho_{\Gamma,0})\in [W^{1,\infty}(\Omega)]^n \times [W^{1,\infty}(\Gamma_1)]^n$ satisfies $\rho_0=\rho_{\Gamma,0}\in \Gamma_1\times (0,T)$. Then, the weak solution of \eqref{wellposedness:schrodinger} belongs to $C^0([0,T];\mathbb{H}_{\Gamma_0}^1)$ with $\pnu v\in L^2(0,T;L^2(\partial \Omega))$. Moreover, there exists $C>0$ such that 
    \begin{align*}
        \max_{t\in [0,T]} \|(v,v_\Gamma)\|_{\mathbb{H}_{\Gamma_0}^1} + \|\pnu v\|_{L^2(0,T;L^2(\partial \Omega))}\leq C \left( \|(h,h_\Gamma)\|_{L^1(0,T;\mathbb{H}_{\Gamma_0}^1)} + \|(v_0,v_{\Gamma,0})\|_{\mathbb{H}_{\Gamma_0}^1} \right).
    \end{align*}
\end{proposition}

The next result establishes a hidden regularity result for the Schr\"odinger equation with dynamic boundary conditions.

\begin{proposition}[Proposition 2.5 \cite{Mercado2023Exact}]
    Under the conditions of the Proposition \ref{proposition:existence:02}, the normal derivative of the solution of \eqref{wellposedness:schrodinger} $(v,v_\Gamma)\in C^0([0,T];\mathbb{H}_{\Gamma_0}^1)$ belongs to $L^2(0,T;L^2(\partial \Omega))$. Moreover, there exists a constant $C>0$ such that
    \begin{align*}
        \|\pnu v\|_{L^2(0,T;L^2(\partial \Omega))}\leq C\left(\|(h,h_\Gamma) \|_{L^1(0,T;\mathbb{H}_{\Gamma_0}^1)} + \|(v_0,v_{\Gamma,0})\|_{\mathbbm{H}_{\Gamma_0}^1} \right). 
    \end{align*}
\end{proposition}

We can also obtain existence and uniqueness results for the one-dimensional Schr\"odinger equation with dynamic boundary conditions as well as a hidden regularity result for the normal derivative. These results can be achieved using multiplier techniques.

\subsection{A Carleman estimate for the Schr\"odinger equation with dynamic boundary conditions}

Consider the function $\psi=\psi(x,t)$ defined in \eqref{def:psi}. We also define, for $\lambda>0$, the following functions 
\begin{align}
\label{weights:theta:varphi}
\theta(x,t):=\dfrac{e^{\lambda \psi(x)}}{(T+t)(T-t)},\quad \varphi(x,t):=\dfrac{\alpha -e^{\lambda \psi(x)}}{(T+t)(T-t)},\quad \forall (x,t)\in \Omega\times (-T,T),
\end{align} 
where $\alpha > \|e^{\lambda \psi}\|_{L^\infty(\Omega)}$.

For $(\vec{\rho}_1,\vec{\rho}_{\Gamma,1})\in L^\infty(-T,T;[\mathbb{L}^\infty]^n)$ and $(\rho_0,\rho_{\Gamma,0})\in L^\infty(-T,T;\mathbb{L}^\infty)$, we introduce the operators
\begin{align}
\label{def:tilde:L}
\tilde{L}(v):=i\pt v-d\Delta v-\vec{\rho}_{1}\cdot \nabla v+\rho_0v\quad \text{ in }\Omega\times (-T,T)
\end{align}
and
\begin{align}
\label{def:tilde:L:gamma}
\tilde{L}(v,v_\Gamma):=i\pt v_\Gamma +d\pnu v -\Delta_\Gamma v_\Gamma -\vec{\rho}_{\Gamma,1}\cdot \nabla_\Gamma v_\Gamma +\rho_{\Gamma,0}v_\Gamma\quad \text{ on }\Gamma_1\times (-T,T).
\end{align}

Moreover, due to the last identity, from now on we shall write $v$ instead of $v_\Gamma$ on $\partial \Omega\times (0,T)$.
Firstly, define the operators 
\begin{align*}
P_1 w=ds^2|\nabla \varphi|^2 w+d\Delta w+i\pt w,\quad P_2 w=ds\Delta\varphi w+2ds\nabla \varphi\cdot \nabla w+is\pt \varphi w,
\end{align*}
and 
\begin{align*}
    Q_1 w=\delta \Delta_\Gamma w+i\pt w,\quad Q_2 w=-ds\pnu \varphi w+is\pt \varphi w.
\end{align*}

\begin{theorem}
\label{thm:Carleman:Schr:Mercado:Morales}
Consider the assumptions {\bf (A1)} and {\bf (A2)} and define the weight functions $\theta$ and $\varphi$ as in \eqref{weights:theta:varphi}. We also consider $(\vec{\rho}_1,\vec{\rho}_{\Gamma,1})\in L^\infty(-T,T;[\mathbb{L}^\infty]^n)$ and $(\rho_0,\rho_{\Gamma,0})\in L^\infty(-T,T;\mathbb{L}^\infty)$. Also, consider the condition \eqref{condition:d:delta}. Then, there exist positive constants $C$, $s_0$ and $\lambda_0$ such that the following estimate holds:
\begin{align}
    \label{Carleman:control}
		\begin{split}
		&\IOTT{e^{-2s\varphi}\left(s^3\lambda^4 \theta^3 |v|^2 +s\lambda \theta |\nabla v|^2 +s\lambda^2 \theta |\nabla \psi \cdot \nabla v|^2 \right)}\\
		&+ \IGTT{e^{-2s\varphi}(s^3\lambda^3 \theta^3 |v|^2 +s\lambda \theta |\pnu v|^2 + s\lambda \theta |\nabla_\Gamma v|^2)}\\
		&+\IOTT{(|P_1 (e^{-s\varphi}v)|^2+|P_2 (e^{-s\varphi}v)|^2)}\\
		&+\IGTT{(|Q_1 (e^{-s\varphi}v)|^2 + |Q_2 (e^{-s\varphi}v)|^2)}\\
		\leq & C \IOTT{e^{-2s\varphi} |\tilde{L}(v)|^2}+ C \IGTT{e^{-2s\varphi}|\tilde{L}_\Gamma(v)|^2}\\
		&+Cs\lambda \int_{-T}^T \int_{\Gamma_\ast} e^{-2s\varphi}\theta |\pnu v|^2\, dS\, dt,\quad \forall s\geq s_0,\, \forall \lambda\geq\lambda_0,\,\forall (v,v_\Gamma)\in \mathcal{V}. 
		\end{split}
\end{align}
where $\mathcal{V}$ is defined by
\begin{align}
    \label{def:mathcal:V}
    \mathcal{V}:=\{(v,v_\Gamma )\in L^2(-T,T;\mathbb{H}_{\Gamma_0}^1)\,:\, (\tilde{L}(v),\tilde{L}_\Gamma(v,v_\Gamma))\in L^2(-T,T;\mathbb{L}^2)\text{ and }\pnu v\in L^2(0,T;L^2(\Gamma_\star))\},
\end{align}
and $\Gamma_\star$ is given in \eqref{def:Gamma:ast}.
\end{theorem}

We point out that the Theorem \ref{thm:Carleman:Schr:Mercado:Morales} is a slight modification of the estimate obtained in \cite{Mercado2023Exact} (see Theorem 1.4 in this reference), where the weight functions are defined on $(0,T)$, and therefore we omit the proof. 


\section{Proof of the stability result}
\label{section:proof:main:thm}

This section is devoted to prove  Theorem \ref{Thm:stability:boundary}. In order to do this, we need to establish a suitable Carleman estimate for the Schr\"odinger operator with dynamic boundary conditions.

\subsection{An auxiliar Carleman estimate}

We start with the following result. 

\begin{theorem}
	\label{Thm:Carleman:ip}
	Consider the assumptions {\bf (A1)} and {\bf (A2)}. Consider $(\vec{\rho}_1,\vec{\rho}_{\Gamma,1})\in L^\infty(-T,T;[\mathbb{L}^\infty]^n)$, $(\rho_{0},\rho_{\Gamma,0})\in L^\infty (-T,T;\mathbb{L}^\infty)$ and consider the operators $\tilde{L}$ and $\tilde{L}_\Gamma$ defined in \eqref{def:tilde:L} and \eqref{def:tilde:L:gamma}, respectively. Assume that $d,\delta >0$ satisfy the condition \eqref{condition:d:delta}. Then, there exist constants $C>0$, $s_0>0$ and $\lambda_0>0$ such that for all $\lambda\geq \lambda_0$ and $s\geq s_0$, the following inequality holds
	\begin{align}
    \label{Carleman:pi}
	\begin{split}
	&s^{3/2}\lambda^{3/2}\left(\int_\Omega e^{-2s\varphi(0)}|v(0)|^2\, dx + \int_{\Gamma_1} e^{-2s\varphi(0)}|v_\Gamma(0)|^2\, dS \right)\\
	\leq &C \IOTT{e^{-2s\varphi}|\tilde{L}(v)|^2} + C\IGTT{e^{-2s\varphi} |\tilde{L}_\Gamma(v,v_\Gamma)|^2}\\
	&+Cs\lambda \int_{-T}^{T}\int_{\Gamma_\star} e^{-2s\varphi}\theta |\pnu v|^2\, dS\, dt, \quad \forall (v,v_\Gamma)\in \mathcal{V},
	\end{split}
	\end{align}
for all $(v,v_\Gamma)\in \mathcal{V}$.
\end{theorem} 

\begin{proof}
As usual, we argue by density, i.e., we consider $(v,v_\Gamma)\in C^\infty((\overline{\Omega}\times \Gamma_1)\times [0,T])$ such that $v=0$ on $\Gamma_0\times (0,T)$ and $v=v_\Gamma$ on $\Gamma_1\times (0,T)$. 
We shall prove that the terms 
\begin{align*}
    \int_\Omega e^{-2s\varphi(\cdot,0)}|v(\cdot,0)|^2\,dx\quad \text{ and } \quad \int_{\Gamma_1} e^{-2s\varphi}(\cdot,0)|v_\Gamma(\cdot,0)|\,dS,
\end{align*}
are bounded by above by  the left-hand side of \eqref{Carleman:control}. To do this, we consider 
		\begin{align*}
		I_1:=&\Im \int_{-T}^0 \int_\Omega P_1(w)\overline{w}\,dx\,dt,\quad I_2:=\Im \int_{-T}^0 \int_{\Gamma_1} Q_1(w)\overline{w}\,dS\, dt
        \end{align*}
        where $\Im$ denotes the imaginary part of a complex number and $w=e^{-s\varphi}v$. After integration by parts, 
        taking into account that  $w=0$ on $\Gamma_0\times (-T,T)$ and $w(\cdot,-T)=0$ in $\Omega$, we deduce that 
        \begin{align}
        \label{computations:I1}
		I_1=&\dfrac{1}{2}\int_\Omega |w(\cdot,0)|^2 \,dx +d\Im \int_{-T}^0 \int_{\Gamma_1} \pnu w \overline{w}\,dS\, dt.
		\end{align}
		
		On the other hand, using the Surface Divergence Theorem 
        \eqref{surface:divergence:theorem} 
        and the fact that $w(\cdot, -T)|_{\Gamma_1} = 0$,
        we see that 
		\begin{align}
        \label{computations:I2}
		I_2=&\dfrac{1}{2} \int_{\Gamma_1} |w(0)|^2 \,dS.
		\end{align}

        Then, adding \eqref{computations:I1} and \eqref{computations:I2}, multiplying by $\frac{1}{2}s^{3/2}\lambda^{3/2}$ and using Young's inequality, we obtain
		\begin{align}
		\label{estimate:wo}
		\begin{split} 
		&\dfrac{1}{2}s^{3/2}\lambda^{3/2}\left(\int_\Omega |w(0)|^2 dx + \int_{\Gamma_1} |w(0)|^2 dS \right)\\
		\leq & 2 \IOTT{(|P_1 (w)|^2 + s^3\lambda^3 |w|^2)}+2 \IGTT{(|Q_1(w)|^2+s^3\lambda^3 |w|^2)}\\
		&+2d \IGTT{(|\pnu w|^2 +s^3\lambda^3 |w|^2)}.
		\end{split}
		\end{align}
		
		By definitions of $w$ and $\varphi$, we see that 
		\begin{align}
		\label{estimate:pnu:w}
		\pnu w=s\lambda e^{-s\varphi} \theta \pnu \psi v + e^{-s\varphi} \pnu v.
		\end{align}
		
		Combining \eqref{Carleman:control}, \eqref{estimate:wo}, \eqref{estimate:pnu:w} and coming back to the original variable, we get 
		\begin{align}
		\nonumber 
		&s^{3/2}\lambda^{3/2}\left(\int_\Omega e^{-2s\varphi (0)}|v(0)|^2 dx + \int_{\Gamma_1} e^{-2s\varphi(0)}|v(0)|^2 dS \right)\\
		\nonumber 
		\leq & C \IOTT{(|P_1(e^{-s\varphi}v)|^2 +s^3\lambda^3 e^{-2s\varphi} |v|^2)}\\
		\nonumber 
		&+C \IGTT{\left( |Q_1(e^{-s\varphi}v)|^2 + s^3 \lambda^3 e^{-2s\varphi} \theta^2 |v|^2 +e^{-2s\varphi} |\pnu v|^2 \right)}\\
		\label{last:estimate:wo}
		\begin{split} 
		\leq & C \IOTT{e^{-2s\varphi} |L(v)|^2}+ C \IGTT{e^{-2s\varphi} |N(v)|^2}\\
		&+Cs\lambda \int_{-T}^T \int_{\Gamma_\star} e^{-2s\varphi} \theta |\pnu v|^2 \,dS\, dt,
		\end{split}
		\end{align}
		for all $s\geq s_0$ and $\lambda\geq \lambda _0$, where the constant $C$ depends on $m,T,d,\delta,\Omega$ and $\Gamma_1$. Finally, \eqref{Carleman:pi} follows easily from \eqref{Carleman:control} and \eqref{last:estimate:wo}. 
\end{proof}


\subsection{Proof of the main result}

Now, we are able to prove the Theorem \ref{Thm:stability:boundary}.
\begin{proof}
 We fix $(q,q_\Gamma)\in \mathbb{L}_{\leq m}^\infty$ and  define 
  \begin{align}
  \label{def:u:u:gamma}
  u=y[p,p_\Gamma]-y[q,q_\Gamma],\text{ and } u_\Gamma=y_\Gamma[p,p_\Gamma]-y_\Gamma[q,q_\Gamma]
  \end{align}
  where $y[p,p_\Gamma]$ and $y[q,q_\Gamma]$ are the solutions of \eqref{original:problem:01} associated to the potentials $(p,p_\Gamma)$ and $(q,q_\Gamma)$, respectively. We also define
  \begin{align}
    \label{def:f:f:gamma}
f=q-p,\quad f_\Gamma= q_\Gamma - p_\Gamma,\quad R=y[q,q_\Gamma]\quad\text{ and }\quad R_\Gamma=y_\Gamma[q,q_\Gamma]. 
  \end{align}
  Then, $(u,u_\Gamma)$ satisfies 
 	\begin{align*}
 	\begin{cases}
 	L(u) +p(x)u=f(x)R(x,t),&\text{ in }\Omega\times (0,T),\\
 	L_\Gamma(u,u_\Gamma)+p_\Gamma(x)u_\Gamma= f_\Gamma(x) R_\Gamma(x,t),&\text{ on }\Gamma_1\times (0,T),\\
 	u\big|_{\Gamma_1}=u_\Gamma,&\text{ on }\Gamma_1\times (0,T),\\
 	u\big|_{\Gamma_0}=0,&\text{ on }\Gamma_0\times (0,T),\\
 	(u,u_\Gamma)(\cdot, 0)=(0,0),&\text{ in }\Omega\times \Gamma_1.
 	\end{cases}
 	\end{align*} 
 	
 	Next, we set $(w,w_\Gamma)=(\pt u,\pt u_\Gamma)$. Then, $(w,w_\Gamma)$ satisfies
 	\begin{align*}
 	\begin{cases} 
 	L(w) +p(x)w=f(x)\pt R(x,t),&\text{ in }\Omega\times (0,T),\\
 	L_\Gamma(w,w_\Gamma)+p_\Gamma (x)w_\Gamma =f_\Gamma(x) \pt R_\Gamma(x,t),&\text{ on }\Gamma_1\times (0,T),\\
 	w\big|_{\Gamma_1}=w_\Gamma,&\text{ on }\Gamma_1\times (0,T),\\
 	w\big|_{\Gamma_0}=0,&\text{ on }\Gamma_0\times (0,T),\\
 	(w,w_\Gamma)(\cdot,0)=(-if(\cdot)R(\cdot,0),-i f_\Gamma(\cdot) R_\Gamma (\cdot,0)),&\text{ in }\Omega\times \Gamma_1. 
 	\end{cases} 
 	\end{align*}
 	
 	Now, we define
  \begin{align*}
      z(x,t)=\begin{cases}
          w(x,t)&\text{ if }(x,t)\in \Omega\times [0,T),\\
          - \overline w(x,-t)&\text{ if }(x,t)\in \Omega\times (-T,0).
      \end{cases}
  \end{align*}
  
We also extend $R$ by $ R(x,t) = \overline R(x,-t)$ for $t<0$, and in an
   analogous  way for $ R_\Gamma$. We denote these extensions by the same symbols. 
   Since $f$ and $R(\cdot,  0)$ are real valued, we have that $w$ satisfies the corresponding 
   system posed in $(-T,T)$.
  In particular, this implies that $(z,z_\Gamma)\in C([-T,T];\mathbb{H}_{\Gamma_0}^1)$ and $\pnu z\in L^2(\partial \Omega\times (-T,T))$. 
  Thus, by regularity assumption \ref{regularity:stability} we deduce that
    \begin{align}
    \label{regularity:R:R:gamma}
    (R,R_\Gamma) \in H^1(-T,T;\mathbb{L}^\infty)
    \end{align}
  	
 	Now, applying Theorem \ref{Thm:Carleman:ip} to $(z,z_\Gamma)$ we obtain
 	\begin{align}
 	\label{estimate:fR:01}
 	\begin{split} 
 	&s^{3/2}\lambda^{3/2}\int_\Omega e^{-2s\varphi(0)}|f R(0)|^2 dx + s^{3/2}\lambda^{3/2} \int_{\Gamma_1}e^{-2s\varphi(0)} |f_\Gamma R_\Gamma(0)|^2 dS\\
 	\leq & C \IOTT{e^{-2s\varphi}|f\pt R|^2} + C \IGTT{e^{-2s\varphi}|f\pt R_\Gamma|^2}\\
 	&+Cs\lambda \int_{-T}^T \int_{\Gamma_\star} e^{-2s\varphi} \theta|\pnu z|^2\, dS\, dt.
 	\end{split} 
 	\end{align}
 	
 	On the other hand, by \eqref{condition:initial:datum:stability} we have
    \begin{align}
    \label{positivity:R:R:gamma}
    |R(\cdot,0)|\geq r_0 >0,\text{ a.e. in }\Omega,\quad  |R_\Gamma (\cdot,0)|\geq r_0>0\text{ a.e on }\Gamma_1.
    \end{align}

Then, from \eqref{regularity:R:R:gamma}  we deduce the existence of $g_\Omega,g_{\Gamma_1}\in L^2(0,T)$ such that 
 	\begin{align}
 	\label{estimate:fR:02}
 	\begin{split} 
 	|\pt R(x,t)|\leq g_\Omega(t)|R(x,0)|,&\quad \forall\, (x,t)\in \Omega\times (-T,T),\\
 	|\pt R_\Gamma(x,t)|\leq g_{\Gamma_1}(t)|R_\Gamma(x,0)|,&\quad \forall (x,t)\in \Gamma_1\times (-T,T).
 	\end{split}
 	\end{align}

 	Moreover, we see that 
 	\begin{align*}
 	-\varphi(x,t)\leq -\varphi(x,0),\quad \forall x\in \overline{\Omega}\times (-T,T).
 	\end{align*}
 	
 	Then, combining \eqref{estimate:fR:01} and \eqref{estimate:fR:02}, we obtain
 	\begin{align}
 	\label{estimate:fR:03}
 	\begin{split} 
 	&s^{3/2}\lambda^{3/2}\int_\Omega e^{-2s\varphi(0)}|f R(0)|^2 dx + s^{3/2}\lambda^{3/2} \int_{\Gamma_1}e^{-2s\varphi(0)} |f_\Gamma R_\Gamma(0)|^2 dS\\
 	\leq & C K \IOT{e^{-2s\varphi(0)}|f|^2 |R(0)|^2} + CK \IGT{e^{-2s\varphi(0)}|f_\Gamma|^2 |R_\Gamma(0)|^2}\\
 	&+Cs\lambda \int_0^T \int_{\Gamma_\star} e^{-2s\varphi}\theta |\pnu z|^2\, dS\, dt,
 	\end{split}
 	\end{align}
 	where the constant $K>0$ satisfies 
 	\begin{align*}
 	\int_0^T |g_\Omega|^2\, dt\leq K,\quad  \int_0^T |g_{\Gamma_1}|^2\, dt\leq K.
 	\end{align*}
 	
 	Thus, taking $s>0$ and $\lambda>0$ large enough if it is necessary we can absorb the first two terms of \eqref{estimate:fR:03}. Additionally, by \eqref{positivity:R:R:gamma} we deduce that 
 	\begin{align}
    \label{estimate:f:f:gamma:final}
 	\begin{split}
 	&s^{3/2}\lambda^{3/2}\int_\Omega e^{-2s\varphi(0)}|f|^2 dx + s^{3/2}\lambda^{3/2} \int_{\Gamma_1}e^{-2s\varphi(0)} |f_\Gamma|^2 dS\\
 	\leq & Cs\lambda \int_0^T \int_{\Gamma_\star} e^{-2s\varphi}\theta |\pnu z|^2\, dS\, dt,
 	\end{split} 
 	\end{align}
 	for all $s\geq s_0$ and $\lambda\geq \lambda_0$. Finally, we fix $s>0$ and $\lambda>0$ and we come back to the original variables in \eqref{def:u:u:gamma} and \eqref{def:f:f:gamma}. Finally, since the weights $\varphi$ and $\theta$ are bounded, we deduce that
 	\begin{align*}
 	\int_\Omega |p-q|^2 dx + \int_{\Gamma_1} |p_\Gamma -q_\Gamma|^2 dS \leq C \int_0^T\int_{\Gamma_\star} |\pnu (\pt y[p,p_\Gamma]-\pt y[q,q_\Gamma])|^2 \,dS \, dt,
 	\end{align*} 
 	and the proof of Theorem \ref{Thm:stability:boundary} is complete.
\end{proof}    
    
 	

\section{Convergence results of CbRec}
\label{section:convergence:CbRec}
Now, we shall give an algorithm to reconstruct the potentials $(p,p_\Gamma)$ on \eqref{original:problem:01}. In order to do that, we introduce an auxiliary functional based on the weights \eqref{weights:theta:varphi} and the Carleman estimate obtained in Theorem \ref{Thm:Carleman:ip}. 

\subsection{An auxiliary functional and some properties}

We fix $(\zeta,\zeta_\Gamma)\in \mathbb{L}^2$, $h\in L^2(0,T;L^2(\Gamma_\ast))$ and choose $s_0>0$ according to Theorem \ref{Thm:Carleman:ip}. Then, for all $s\geq s_0$, we introduce the functional 
\begin{align}
\begin{split} 
\label{def:J}
J[\zeta,\zeta_\Gamma,h](u,u_\Gamma):=&\dfrac{1}{2s}\IOT{e^{-2s\varphi} |N(u)-\zeta|^2}+\dfrac{1}{2s}\IGT{e^{-2s\varphi}|N_\Gamma(u,u_\Gamma)-\zeta_\Gamma|^2}\\
&+\dfrac{1}{2}\int_0^T\int_{\Gamma_\star}{e^{-2s\varphi} |\pnu u-h|^2 }\,dS\, dt,\quad \forall\, (u,u_\Gamma)\in \mathcal{W},
\end{split}
\end{align}
where $N(u)=L(u)+p(x)u$ and $N(u,u_\Gamma)=L_\Gamma(u,u_\Gamma)+p_\Gamma(x)u_\Gamma$, with $L$ and $L_\Gamma$ defined in \eqref{def:operator:L} and \eqref{def:operator:L:gamma} and
\begin{align*}
\mathcal{W}:=\left\{ (y,y_\Gamma)\in L^2(0,T;\mathbb{H}_{\Gamma_0}^1)\,:\, (N(y),N_\Gamma (y,y_\Gamma))\in L^2(0,T;\mathbb{L}^2),\, \pnu y\in L^2(0,T;L^2(\Gamma_\star))\right\},
\end{align*}
endowed by the norm
\begin{align}
\label{def:norm:W}
\begin{split} 
\|(u,u_\Gamma)\|_{\mathcal{W}}^2 :=&\dfrac{1}{s}\IOT{e^{-2s\varphi} |N(u)|^2}+ \dfrac{1}{s}\IGT{e^{-2s\varphi} |N_\Gamma(u,u_\Gamma)|^2}\\
&+\int_0^T\int_{\Gamma_\ast}{e^{-2s\varphi}|\pnu u|^2\, dS\, dt}.
\end{split}
\end{align}

We point out that $J$ depends on the parameter $s>s_0$. However, to simply notation, we just write $J$ instead of $J_s$. From now on, we shall consider the unconstrained problem 
\begin{align}
    \label{problem:min:J}
    \begin{cases}
        \text{Minimize }&J[\zeta,\zeta_\Gamma,h](u,u_\Gamma)\\
        \text{Subject to }&(u,u_\Gamma)\in \mathcal{W}.
    \end{cases}
\end{align}
\begin{theorem}\label{Thm:properties:J}
	For $(\zeta,\zeta_\Gamma)\in \mathbb{L}^2$, $h\in L^2(0,T;L^2(\Gamma_\star))$ and $s_0$ given by Theorem \ref{Thm:Carleman:ip}, we consider the functional $J$ defined on $\mathcal{W}$, for all $s\geq s_0$. Then,
	\begin{enumerate}[(a)]
		\item The problem \eqref{problem:min:J} admits a unique minimizer $(u^*,u_\Gamma^*)\in \mathcal{W}$. Moreover, there exists a constant $C>0$ independent of $s$ such that
		\begin{align}
		\label{estimate:yast}
		\begin{split} 
		\|(u^*,u_\Gamma^*)\|_{\mathcal{W}}^2 \leq &\dfrac{C}{s} \IOT{e^{-2s\varphi}|\zeta|^2}+\dfrac{C}{s} \IGT{e^{-2s\varphi}|\zeta _\Gamma|^2}\\
		&+C\int_0^T\int_{\Gamma_\star}{e^{-2s\varphi}|h|^2}\,dS \, dt.
		\end{split}
		\end{align}
		\item The minimizer $(u^*,u_\Gamma^*)\in \mathcal{W}$ satisfies the associated Euler-Lagrange equation
		\begin{align}
        \label{EL:eq}
		\begin{split}
		&\dfrac{1}{s}\Re \IOT{e^{-2s\varphi} (N(u^*)-\zeta)\overline{N(u)}} + \dfrac{1}{s}\Re \IGT{e^{-2s\varphi} (N_\Gamma(u^*,u_\Gamma^*)-\zeta_\Gamma)\overline{N_\Gamma(u,u_\Gamma)}}\\
		&+\Re \int_0^T\int_{\Gamma_\star}{e^{-2s\varphi} (\pnu u^*-h)\overline{\pnu u}}\,dS\, dt=0, \quad \forall (u,u_\Gamma)\in \mathcal{W},
		\end{split}
		\end{align}
		where $\Re$ denotes the real part of a complex number. 
		\item If $(\zeta^j,\zeta_\Gamma^j)\in L^2(0,T;\mathbb{L}^2)$, $h\in L^2(0,T;L^2(\Gamma_\star))$ and $(u^{*,j},u_\Gamma^{*,j})\in \mathcal{W}$ is the corresponding minimizer of $J[\zeta^j,\zeta_\Gamma^j,h]$ for $j\in \{a,b\}$, then we have the following estimate
		\begin{align}
		\label{estimate:ya:yb}
		\begin{split} 
		&s^{3/2}\int_\Omega e^{-2s\varphi(\cdot,0)}|(u^{*,a}-u^{*,b})(\cdot,0)|^2\,dx+s^{3/2}\int_{\Gamma_1}e^{-2s\varphi(\cdot,0)}|(u^{*,a}_\Gamma - u^{*,b}_\Gamma)(\cdot,0)|^2 \,dS\\
		\leq &C\IOT{e^{-2s\varphi}|\zeta^a-\zeta^b|^2}+ C\IGT{e^{-2s\varphi}|\zeta_\Gamma^a - \zeta_{\Gamma}^b|^2},
		\end{split}
		\end{align}
		for some constant $C>0$ independent of $s$.
	\end{enumerate}
\end{theorem}

\begin{proof}
    \begin{enumerate}[(a)]
 		\item Clearly, the functional $J$ is clearly continuous and strictly convex. In order to prove that $J$ is coercive, we point out that
 		\begin{align}
 		\nonumber 
 		&J[\zeta,\zeta_\Gamma,h](y,y_\Gamma)\\
 		\label{estimate:coercive:01}
 		\begin{split} 
 		=&\dfrac{1}{2}\|(y,y_\Gamma)\|_{\mathcal{W}}^2+\dfrac{1}{2s}\IOT{e^{-2s\varphi}|\zeta|^2}+\dfrac{1}{2s}\IGT{e^{-2s\varphi}|\zeta_\Gamma|^2}\\
 		&+\dfrac{1}{2}\int_0^T\int_{\Gamma_\star}{e^{-2s\varphi}|h|^2}dS dt
 		-\dfrac{1}{s}\Re \IOT{e^{-2s\varphi}\zeta \overline{N(y)}}\\
 		&-\dfrac{1}{s}\Re \IGT{e^{-2s\varphi}\zeta_\Gamma \overline{N_\Gamma(y,y_\Gamma)}}-\Re \int_0^T\int_{\Gamma_\star}{e^{-2s\varphi}\pnu y \overline{h}}dS dt
 		\end{split}
 		\\
 		\label{estimate:coercive:02}
 		\begin{split} 
 		\geq &\dfrac{1}{4}\|(y,y_\Gamma)\|_\mathcal{W}^2 -\dfrac{1}{s}\IOT{e^{-2s\varphi}|\zeta|^2}-\dfrac{1}{s}\IGT{e^{-2s\varphi}|\zeta_\Gamma|^2}\\
 		&-\int_0^T\int_{\Gamma_\ast}{e^{-2s\varphi}|h|^2}\,dS\, dt,
 		\end{split} 
 		\end{align}
 		where in \eqref{estimate:coercive:02} we have used Young's inequality in the last step and $\|\cdot\|_{\mathcal{W}}$ is given in \eqref{def:norm:W}. This implies that $J$ is coercive and therefore $J$ admits a unique minimizer $(y^*,y_\Gamma^*)\in \mathcal{W}$. In order to prove \eqref{estimate:yast}, since $J[\zeta,\zeta_\Gamma,h](y^*,y_\Gamma^*)\leq J[\zeta,\zeta_\Gamma,h](0,0)$ and \eqref{estimate:coercive:01} we have 
 		\begin{align*}
 		\begin{split} 
 		&\dfrac{1}{2}\|(y^*,y_\Gamma^*)\|_\mathcal{W}^2\\
 		\leq &\dfrac{1}{s}\Re \IOT{e^{-2s\varphi}\zeta \overline{N(y^*)}}+\dfrac{1}{s}\Re \IGT{e^{-2s\varphi}\zeta_\Gamma \overline{N_\Gamma(y^*,y_\Gamma^*)}}\\
 		&+\dfrac{1}{s}\Re \int_0^T\int_{\Gamma_\star}{e^{-2s\varphi}h \overline{\pnu y^*}}\,dS\, dt.
 		\end{split}
 		\end{align*}
 		Then, by Young's inequality, for all $\epsilon>0$, there exists a constant $C=C(\epsilon)$ such that 
 		\begin{align*}
 		&\dfrac{1}{2}\|(y^*,y_\Gamma^*)\|_\mathcal{W}\\
 		\leq & \dfrac{C}{s} \IOT{e^{-2s\varphi}|\zeta|^2}+ \dfrac{C}{s}\IGT{e^{-2s\varphi}|\zeta_\Gamma|^2} \\
 		&+\int_0^T\int_{\Gamma_\star}{e^{-2s\varphi}|h|^2}\,dS\, dt + \epsilon \|(y^*,y_\Gamma^*)\|_\mathcal{W}^2.
 		\end{align*}
 		Thus, taking $\epsilon>0$ small enough we conclude the proof of \eqref{estimate:yast}.
 		\item Direct from the fact that $J$ has a unique minimizer.
 		\item By (b), we have the following inequalities:
 		\begin{align}
 		\label{prop:c:01}
 		\begin{split}
 		&\dfrac{1}{s}\Re \IOT{e^{-2s\varphi} (N(y^{*,a})-\zeta^a)\overline{N(y)}}+\dfrac{1}{s}\Re \IGT{e^{-2s\varphi} (N_\Gamma(y^{*,a})-\zeta_\Gamma^a)\overline{N_\Gamma(y,y_\Gamma)}}\\
 		&+\Re \int_0^T\int_{\Gamma_\star}{e^{-2s\varphi}(\pnu y^{*,a}-h)\overline{\pnu y}}\,dS\, dt=0,\quad \forall (y,y_\Gamma)\in \mathcal{W},
 		\end{split}
 		\end{align}
 		and
 		\begin{align}
 		\label{prop:c:02}
 		\begin{split}
 		&\dfrac{1}{s}\Re \IOT{e^{-2s\varphi} (N(y^{*,b})-\zeta^b)\overline{N(y)}}+\dfrac{1}{s}\Re \IGT{e^{-2s\varphi} (N_\Gamma(y^{*,b})-\zeta_\Gamma^b)\overline{N_\Gamma(y,y_\Gamma)}}\\
 		&+\Re \int_0^T\int_{\Gamma_\star}{e^{-2s\varphi}(\pnu y^{*,b}-h)\overline{\pnu y}}\,dS\, dt=0,\quad \forall (y,y_\Gamma)\in \mathcal{W}.
 		\end{split}
 		\end{align}
 		Subtracting \eqref{prop:c:01} and \eqref{prop:c:02} and taking $(y,y_\Gamma)=(y^{*,a}-y^{*,b},y_\Gamma^{*,a}-y_\Gamma^{*,b})$, we deduce that 
 		\begin{align*}
 		\|(y,y_\Gamma)\|_\mathcal{W}^2=&\dfrac{1}{s}\Re \IOT{e^{-2s\varphi} (\zeta^a-\zeta^b)\overline{N(y)}}\\
 		&+ \dfrac{1}{s}\Re \IGT{e^{-2s\varphi}(\zeta_\Gamma^a-\zeta_\Gamma^b)\overline{N_\Gamma(y,y_\Gamma)}},
 		\end{align*}
 		and by Young's inequality it is easy to see that
 		\begin{align}
 		\label{eq:asdf}
 		\|(y,y_\Gamma)\|_\mathcal{W}^2 \leq \dfrac{C}{s}\IOT{e^{-2s\varphi}|\zeta^a -\zeta^b|^2}+\dfrac{C}{s}\IGT{e^{-2s\varphi}|\zeta_\Gamma^a -\zeta_\Gamma^b|^2}.
 		\end{align}
 		Finally, combining \eqref{Carleman:pi} and \eqref{eq:asdf}, we get \eqref{estimate:ya:yb}. This concludes the proof of Theorem \ref{Thm:properties:J}.
 	\end{enumerate}
\end{proof}

\subsection{A CbRec type algorithm}
We now present an algorithm designed to reconstruct the potentials  \( p \) and \( p_{\Gamma} \). To do this, we fix $m>0$ and define the space $\mathbb{L}_{\leq m}^\infty$ defined in \eqref{def:L:infty:m}. Then, the algorithm reads as follows:

\begin{algorithm}[!h]
\caption{Reconstruction algorithm for coefficients $p$ and $p_\Gamma$}\label{alg:cap}
\label{Algorithm:reconstruction}
\textbf{Initialization:} 
\begin{itemize}
\item $(p^0, p_{\Gamma}^0) = (0, 0)$, or any guess $(p^0, p_{\Gamma}^0) \in \mathbb{L}_{\leq m}^\infty $.
\end{itemize}
\textbf{Iteration:} From $k$ to $k+1$.
\begin{itemize}
    \item Step 1: Given $(p^k, p_\Gamma^k)$, we set
    \begin{alignat*}{1}
    h^k=\pt \left(\pnu y[p^k,p_\Gamma^k]-\pnu y[p,p_\Gamma] \right)
    \end{alignat*}
    where $(y,y_\Gamma)[p^k,p^k_\Gamma]$ and $(y,y_\Gamma)[p,p_\Gamma])$ are the solutions of \eqref{original:problem:01} associated to the potentials $(p^k,p^k_\Gamma)$ and $(p,p_\Gamma)$, respectively.
    \item Step 2: Find the minimizer $(u^{\ast,k},u_\Gamma^{\ast,k})$ of the unconstrained problem
    \begin{align*}
\begin{cases} 
        \text{Minimize }&J[0,0,h^k](u,u_\Gamma),\\
\text{Subject to }&(u,u_\Gamma)\in \mathcal{W}
\end{cases}
    \end{align*}
    \item Step 3: Set
    \begin{alignat}{1}
\label{def:pk1}
\tilde{p}^{k+1} = p^k + \frac{\partial_t u^{\ast,k} (0)}{y_0} \quad \text{in } \Omega,\quad 
    \tilde{p}_{\Gamma}^{k+1} = p_{\Gamma}^k + \frac{\partial_t u_{\Gamma}^{\ast,k} (\cdot,0)}{y_{\Gamma, 0}} \quad \text{in } \Gamma_1,
    \end{alignat}
    \item Step 4: Finally, consider $p^{k+1}=\mathcal{T}(\tilde{p}^{k+1})$ and $p_\Gamma^{k+1}=\mathcal{T}_\Gamma(\tilde{p}_\Gamma ^{k+1})$, where the operators $\mathcal{T}$ and $\mathcal{T}_\Gamma$ are given by
    \begin{align*}
        \mathcal{T}(p):=\begin{cases}
            p&\text{ if }\lvert p\rvert \leq m,\\
            m \dfrac{p}{\lvert p\rvert} 
            &\text{ if }\lvert  p\rvert >m
        \end{cases}
        \text{ and }
        \mathcal{T}_\Gamma(p_\Gamma):=\begin{cases}
            p_\Gamma&\text{ if }\lvert p_\Gamma \rvert \leq m,\\
             m  \dfrac{p_{\Gamma}}{\lvert p_\Gamma\rvert}
             &\text{ if }\lvert p_\Gamma\rvert>m.
        \end{cases}
    \end{align*}
\end{itemize}	
\end{algorithm}	

\FloatBarrier

Using  Theorem \ref{Thm:properties:J}, in the next result  we prove the convergence of the Algorithm \ref{Algorithm:reconstruction}.
\begin{theorem}
    \label{Thm:convergence:Algorithm:01}
    Assume the hypotheses of Theorem \ref{Thm:properties:J}. Additionally, suppose that
    \begin{align*}
        (u[p,p_\Gamma],u_\Gamma[p,p_\Gamma])\in H^2(0,T;\mathbb{L}^\infty).
    \end{align*}

    Let $m>0$, $(p,p_\Gamma)\in \mathbb{L}_{\leq m}^\infty$ and for each $k\in \mathbb{N}$, consider $(p^k,p_\Gamma^k)\in \mathbb{L}_{\leq m}^\infty$. Then, there exist a constant $C_0>0$ and $s_0>0$ such that for all $s\geq s_0$ and $k\in \mathbb{N}$, we have 
    \begin{align}
        \label{convergence:ineq}
        \begin{split} 
        &\int_\Omega e^{-2s\varphi(\cdot,0)} |p^{k+1}-p|^2\,dx + \int_{\Gamma_1} e^{-2s\varphi(\cdot,0)} |p_\Gamma^{k+1}-p_\Gamma|^2\,dS \\
        \leq &\dfrac{C_0}{s^{3/2}}\int_\Omega e^{-2s\varphi(\cdot,0)} |p^k-p|^2\,dx + \dfrac{C_0}{s^{3/2}}\int_{\Gamma_1} e^{-2s\varphi(\cdot,0)}|p_\Gamma^{k}-p_\Gamma|^2\,dx.
        \end{split}
    \end{align}
    This implies that
    \begin{align}
\label{convergence:ineq:02}
        \|(p^{k+1}-p,p_\Gamma^{k+1}-p_\Gamma)\|_{\mathbb{L}^2}\leq C_s \|(p^{k}-p,p_\Gamma^k-p_\Gamma)\|_{\mathbb{L}^2},\quad \forall k\in \mathbb{N},
    \end{align}
    where $C_s$ is given explicitly by
    \begin{align*}
        C_s:=C_0 \dfrac{\displaystyle \max_{x\in \overline{\Omega}}e^{-2s\varphi(x,0)} }{\displaystyle \min_{x\in \overline{\Omega}} e^{-2s\varphi(x,0)}}.
    \end{align*}
    In particular, \eqref{convergence:ineq:02} implies that the Algorithm defined in
    \ref{Algorithm:reconstruction} converges for all $s$ sufficiently large.
\end{theorem}

\begin{remark}
    As we said before, the principal significance of the Theorem \ref{Thm:convergence:Algorithm:01} is that we can use the Algorithm \ref{Algorithm:reconstruction} to reconstruction simultaneously the potentials $p$ and $p_\Gamma$ in \eqref{original:problem:01}. Moreover, since the weights used in functional \ref{def:J} does not blow up as $t\to T$, it is expected that our findings can be implemented numerically.    
\end{remark}

\begin{proof}[Proof of Theorem \ref{Thm:convergence:Algorithm:01}]
In this section, we prove the convergence of the Algorithm \ref{Algorithm:reconstruction}. To do this, let $k\in\mathbb{N}$ and consider $u^k:=\pt (y[p^k,p_\Gamma^k]-y[p,p_\Gamma])$ and $u_\Gamma^k:=\pt (y_\Gamma[p^k,p_\Gamma^k]-y_\Gamma[p,p_\Gamma])$. Then, $(u^k,u_\Gamma^k)$ is a solution of 
\begin{align}
\label{conv:00}
\begin{cases}
    N(u^k) =(p-p^k)\pt R(x,t)&\text{ in }\Omega\times (0,T),\\
N_\Gamma(u^k,u_\Gamma^k)=(p_\Gamma -p_\Gamma^k)\pt R_\Gamma (x,t)&\text{ in }\Gamma_1\times (0,T),\\
u^k\big|_{\Gamma_1}=u_\Gamma^k&\text{ on }\Gamma_1\times (0,T),\\
u^k=0&\text{ on }\Gamma_0\times (0,T),\\
(u^k(\cdot,0),u_\Gamma^k(\cdot,0))=(-i(p-p^k)R(\cdot,0),-i(p_\Gamma-p_\Gamma^k)R_\Gamma(\cdot,0))&\text{ in }\Omega\times \Gamma_1,
\end{cases}
\end{align}
where $R$ and $R_\Gamma$ are given by $R:= y[p,p_\Gamma]$ and $R_\Gamma:= y_\Gamma[p,p_\Gamma]$, respectively. Then, we set
\begin{align}
\label{conv:01}
h^k:=\pnu y^k.
\end{align}

Note that $y^k\in \mathcal{W}$. Therefore, by \eqref{conv:01} the solution $u^k$ of \eqref{conv:00} satisfies the Euler-Lagrange equation associated to the functional $J[\zeta^k,\zeta_\Gamma^k,h^k]$ in \eqref{EL:eq} with $\zeta^k=(p-p^k)\pt R(x,t)$ and $\zeta_\Gamma^k=(p_\Gamma-p_\Gamma^k)\pt R$. Since $J[\zeta^k,\zeta_\Gamma^k,h^k]$ admits a unique minimizer, $u^k$ corresponds to the minimum of $J[\zeta^k,\zeta_\Gamma^k,h^k]$.

Now, let $(u^{\ast,k},u_\Gamma^{\ast,k})$ be the minimizer of $J[0,0,h^k]$. Then, from inequality \eqref{estimate:ya:yb} we obtain
\begin{align}
\label{conv:02}
\begin{split}
&s^{3/2}\int_\Omega e^{-2s\varphi(\cdot,0)} |u^{\ast,k}(\cdot,0)-u^{k}(\cdot,0)|^2\,dx + s^{3/2}\int_{\Gamma_1} e^{-2s\varphi(\cdot,0)} |u_\Gamma^{\ast,k}(\cdot,0)- u_\Gamma^{k}(\cdot,0)|^2\,dS  \\
    \leq & C \IOT{e^{-2s\varphi} (|p-p^{k}|\pt R)^2} + C\IGT{e^{-2s\varphi}(|p_\Gamma - p_\Gamma^k|\pt R_\Gamma)^2}.
\end{split}
\end{align}

From \eqref{def:pk1} and \eqref{conv:00}, we deduce that 
\begin{align}
\label{conv:03}
    u^{\ast,k}(\cdot,0)=(\tilde{p}^{k+1}-p^k)y_0,\quad u^{\ast,k}(\cdot,0)=(\tilde{p}_\Gamma^{k+1}-p_\Gamma)y_{\Gamma,0}.
\end{align}

Substituting \eqref{conv:03} into \eqref{conv:02} and using the condition \eqref{condition:initial:datum:stability} we deduce that 
\begin{align} 
\label{conv:04}
\begin{split} 
&s^{3/2}\int_\Omega e^{-2s\varphi(\cdot,0)}|\tilde{p}^{k+1}-p|^2\,dx +s^{3/2}\int_{\Gamma_1} e^{-2s\varphi(\cdot,0)}|(p_\Gamma^{k}-p_\Gamma)|^2\, dS\\
    \leq & C\IOT{e^{-2s\varphi} (|p-p^{k}|\pt R)^2} + C\IGT{e^{-2s\varphi}(|p_\Gamma - p_\Gamma^k|\pt R_\Gamma)^2}.
\end{split}
\end{align}

On the other hand, since $\mathcal{T}$ and $\mathcal{T}_\Gamma$ are Lipschitz continuous functions and satisfy $\mathcal{T}(p)=p$ and $\mathcal{T}(p_\Gamma)=p_\Gamma$, we have 
\begin{align*}
    |\tilde{p}^{k+1}-p|\geq |\mathcal{T}(\tilde{p}^{k+1}-T(p))|=|p^{k+1}-p|
\end{align*}
and
\begin{align*}
    |\tilde{p}_\Gamma^{k+1}-p_\Gamma|\geq |\mathcal{T}(\tilde{p}_\Gamma^{k+1}-T(p_\Gamma))|=|p_\Gamma^{k+1}-p_\Gamma|.
\end{align*}

Finally, since $-\varphi(x,t)\leq -\varphi(x,0)$, for all $x\in \overline{\Omega}$ and $t\in (0,T)$, $(\pt R,\pt R_\Gamma)\in L^2(0,T;\mathbb{L}^\infty)$, we conclude \eqref{convergence:ineq}. The inequality \eqref{convergence:ineq:02} can be obtained directly of \eqref{convergence:ineq} since $\varphi(\cdot,0)$ is bounded. This ends the proof of the Theorem \ref{Thm:convergence:Algorithm:01}.
\end{proof}

\section{Further comments and concluding remarks}
\label{section:further:comments}

In this paper, we have presented the coefficient inverse problem {\bf (CIP)} for a Schr\"odinger equation with dynamic boundary conditions. We have provided a Lipschitz stability result and a reconstruction algorithm. The stability result was obtained applying the Bukhgeim-Klibanov method, while the reconstruction algorithm is inspired on the CbRec Algorithm proposed in \cite{Baudouin2013Global}. Besides, both results strongly depends on a suitable Carleman estimate for the Schr\"odinger operator with dynamic boundary conditions with observations on the normal derivative obtained for these purposes.

The assumption {\bf (A1)} plays an important role in our results. The strong convexity condition on the $\Omega_1$ is given to define a weight function which is constant on the boundary $\Gamma_1$. To the best of our knowledge, the case where $\Omega_1$ is non strictly convex has not been considered yet in the context of the wave and Schr\"odinger equation with dynamic boundary conditions. 

Our results depends also of the condition $\delta>d$, where it is used to deduce the Carleman estimate for the Schr\"odinger operator. 
We mention that an analogous hypothesis  also appeared in context of controllability of the wave equation with acoustic boundary conditions, see for instance \cite{Baudouin2022A} and \cite{chorfi2024lipschitz}. In particular, in the case of an annulus, it is proved in \cite{Baudouin2022A} that in the case $\delta<d$, the associated adjoint problem is not observable at any time (see Theorem 2.4 in that reference).
However, as far as we know, the case $\delta=d$ remains open  for the wave equation. 
The same questions can be considered for the corresponding Schr\"odinger system. 

One of the main advantages of the proposed algorithm, in contrast to the Tikhonov regularization techniques, is the fact that it converges to the exact potential $(p,p_\Gamma)$ independent of the initial guess. The numerical implementation of CbRec type algorithm for Schr\"odinger equations is, as far as we known, unexplored. Even the numerical Schr\"{o}dinger equation with dynamic boundary conditions has not been studied, as far as the authors know. However, it seems not to be a problem to show in numerical experiments that its behavior is sufficiently good, while the main difficulties are shown in the implementation of the functional minimization step, which is not surprising due to previous works in CbRec type algorithms for other equations, such as in \cite{Baudouin2017Convergent,Boulakia2021Numerical}, where additional terms have to be added to the functional to be minimized, and filters have to be applied in intermediate steps of the algorithm. Both additional tasks in the algorithm have been considered by the authors of the  mentioned works due to regularity issues and for dealing with noisy measurements and the propagation of them in numerical differentiation. These difficulties seem to be shown even in the case of Dirichlet boundary conditions. Since numerical aspects of the implementation seem to be a separate subject from the scope of this article, this will be investigated in a forthcoming work.


\section*{Acknowledgments}
The second author has been partially supported by Fondecyt 1211292 and ANID Millennium Science Initiative Program, Code NCN19-161. The third author has been funded under the Grant QUALIFICA by Junta de Andaluc\'ia grant number QUAL21 005 USE.


\appendix
\section{Carleman estimate for the 1-D Schr\"odinger equation with dynamic boundary conditions}

In this section, we deduce a Carleman estimate for the one-dimensional Schr\"odinger equation with dynamic boundary conditions. We set $\Omega:=(0,\ell)$ and $T>0$. Then, we consider the following problem:
\begin{equation*}
 \begin{cases}
i \partial_t y+d\Delta y + p y = g & \forall\, (x,t)\in \Omega \times (0, T), \\
i \dot{y}_\Gamma(t) - d\partial_x y(0,t) + p_\Gamma y_\Gamma(t) = g_\Gamma(t) & \forall\, t\in  (0, T), \\
y(0,t) = y_\Gamma(t) & \forall \,t\in  (0, T), \\
y(\ell,t) = 0 &\forall\, t\in (0, T), \\
(y, y_\Gamma) (\cdot , 0 ) = (y_0 , y_{\Gamma, 0}) &\forall\, x\in  \Omega,
\end{cases}
\end{equation*}

Given a point $x_1=-a$, with $a>0$, set $\psi(x):=|x-x_1|^2$ for each $x\in \overline{\Omega}$ and for $\lambda>0$ we define the weight functions
\begin{align}
    \label{weights:1-D}
    \theta(x,t):=\dfrac{e^{\lambda \psi(x)}}{(T+t)(T-t)},\quad \varphi(x,t):=\dfrac{\alpha-e^{\lambda \psi(x)}}{(T+t)(T-t)}\quad \forall (x,t)\in \overline{\Omega}\times (-T,T),
\end{align}
with $\alpha>\|e^{\lambda \psi}\|_{L^\infty(\Omega)}$.

For $d>0$ and $s>0$, we introduce the operators:
\begin{align}
    \label{def:P1:P2}
    P_1w=ds^2(\px \varphi)^2 w+d\pxx w +i\pt w,\quad P_2w=ds\pxx \varphi w+2ds\px \varphi \px w+is\pt w,
\end{align}
and
\begin{align}
    \label{def:Q1:Q2}
    Q_1w:=i\pt w,\quad Q_2w:=-ds\px \varphi w+is\pt \varphi w.
\end{align}

Besides, for $q_0,q_1\in L^\infty(\Omega\times (-T,T))$ and $q_{\Gamma,0}\in L^\infty(-T,T)$, consider the operators
    \begin{align}
        \label{def:tilde:L:Lgamma}
        \begin{split} 
        \tilde{L}(v):=&i\pt v+d\pxx v+q_1\px v+q_0 v,\\        \tilde{L}_\Gamma(v,v_\Gamma):=&\dot{v}_{\Gamma_0}(t)-\px v(0,t)+q_{\Gamma,0}(t)v_\Gamma(t).
        \end{split} 
    \end{align}

Then, we have the following Carleman estimate:

\begin{lemma}
    \label{lemma:Carleman:1D}
    Let $\theta$ and $\varphi$ be the functions defined in \eqref{weights:1-D}.
    There exist $C>0$, $s_0>0$ and $\lambda_0>0$ such that for all $\lambda\geq \lambda_0$ and $s\geq s_0$,
    \begin{align}
        \label{Carleman:1D}
        \begin{split}
            &\int_{-T}^T \int_\Omega e^{-2s\varphi}(s^3\lambda^4\theta^3|v|^2+s\lambda^2 \theta |\px v|^2)\,dx\,dt\\
            &+\int_{-T}^{T} e^{-2s\varphi(0,\cdot)} (s^3\lambda^3 \theta(0,\cdot)^3|v(0,\cdot)|^2+s\lambda \theta(0,\cdot)|\px v(0,\cdot)|^2)dt\\
            \leq & C\int_{-T}^T \int_\Omega e^{-2s\varphi}|\tilde{L}(v)|^2\,dx\,dt + C\int_{-T}^T e^{-2s\varphi(0,\cdot)} |\tilde{L}_\Gamma(v(0,\cdot),v_{\Gamma_0}(\cdot))|^2\,dt\\
            &+Cs\lambda \int_{-T}^T e^{-2s\varphi(\ell,\cdot)} \theta(\ell,\cdot)|\px v(\ell,\cdot)|^2\,dt,
        \end{split}
    \end{align}
    for all $(v,v_\Gamma)\in L^2(-T,T;\mathbb{H}_{\Gamma_0}^1)$ such that $L(v)\in L^2(\Omega\times (-T,T))$ and $L_\Gamma(v,v_\Gamma)\in L^2(-T,T)$ with $\px v(\ell,\cdot)\in L^2(-T,T)$.

\end{lemma}

\begin{proof}
    We argue by density arguments. Therefore, we just write $v_{\Gamma_0}(t)=v(0,t)$ for all $t>0$. The proof of Lemma \ref{lemma:Carleman:1D} can be divided into four steps:
    
    \noindent $\bullet$\textit{Step 1:} We compute the terms 
    \begin{align*}
        Pw=e^{-s\varphi}\tilde{L}(e^{s\varphi}w),\quad Q w=e^{s\varphi}L_\Gamma(e^{s\varphi}w).
    \end{align*}

    Straightforward computations show that 
    \begin{align*}
        Pw=P_1w+P_2w,\quad Qw=Q_1w+Q_2w+R_\Gamma w,
    \end{align*}
    where $P_1,P_2,Q_1$ and $Q_2$ are defined in \eqref{def:P1:P2} and \eqref{def:Q1:Q2}, respectively, and $R_\Gamma=-d\px w$. 

    Then, we have the following identity
    \begin{align}
        \label{estimate:P1P2:Q1:Q2:R}
        \begin{split}
            &\int_{-T}^T \int_\Omega (|P_1 w|^2+|P_2 w|^2)\,dx\,dt + \int_{-T}^T (|Q_1 w|^2+|Q_2 w|^2)\,dt\\
            &+2\Re \int_{-T}^T \int_\Omega P_1 w\overline{P_2 w}\,dx\,dt +2\Re \int_{-T}^T Q_1w\overline{Q_2w}\,dt\\
            =&\int_{-T}^T \int_\Omega |Pw|^2\,dx\,dt +\int_{-T}^T |Rw-R_\Gamma w|^2\,dt
        \end{split}
    \end{align}

    \noindent $\bullet$ \textit{Step 2:} Estimates in $\Omega\times (-T,T)$. In this step, we compute the terms  
    \begin{align*}
        \Re \int_{-T}^T \int_\Omega P_1w\overline{P_2w}\,dx\,dt=\sum_{j=1}^3 \sum_{k=1}^3 I_{jk}.
    \end{align*}

    In the following, we shall use the formulas:
    \begin{align*}
        \px \phi =-\lambda\theta \psi',\quad \px \theta=\lambda \theta \psi',\quad \pxx \varphi =-\lambda^2 \theta |\psi'|^2-\lambda \theta \psi''\quad \text{ in }\overline{\Omega}\times (-T,T).
    \end{align*}

    The first term is given by 
    \begin{align*}
        I_{11}=&d^2s^3\int_{-T}^T \int_\Omega \pxx \varphi (\px \varphi)^2 |w|^2\,dx\,dt \\
        =&-d^2s^3\lambda^3 \int_{-T}^T\int_\Omega (\lambda | \psi'|^2 +1)( \psi')^2\theta^3 |w|^2\,dx\,dt.
    \end{align*}

    Moreover, integration by parts yields
    \begin{align*}
        I_{12}=&2s^2s^3\Re\int_{-T}^T\int_\Omega (\px \varphi)^3 w\px \overline{w}\,dx\,dt\\
        =&3d^2s^3\lambda^3 \int_{-T}^T\int_\Omega (\lambda | \psi'|^2 +1)( \psi')^2\theta^3 |w|^2\,dx\,dt +d^2s^3\lambda^3 \int_{-T}^T (\psi')^3\theta(0,\cdot)^3 |w(0,\cdot)|^2\,dt. 
    \end{align*}

    On the other hand, using that $\Re(iz)=-\Im(z)$, for all $z\in \mathbb{C}$, we have 
    \begin{align*}
        I_{13}=-ds^3\lambda^2 \Im \int_{-T}^T \int_\Omega (\px \varphi)^2 \pt \varphi |w|^2\,dx\,dt=0.
    \end{align*}

    If $\gamma>0$ is a constant which satisfies 
    \begin{align*}
         \psi', \psi''\geq \gamma >0,
    \end{align*}
    then we have 
    \begin{align}
        \label{P1}
        \sum_{k=1}^3 I_{1k}\geq 2d^2\gamma^4 s^3\lambda^4 \int_{-T}^T\int_\Omega \theta^3 |w|^2\,dx\,dt +d^2\gamma^3 s^3\lambda^3 \int_{-T}^T \theta^3(0,\cdot)|w(0,\cdot)|^2\,dt.
    \end{align}

    Now, the term $I_{21}$ can be computed as follows
    \begin{align*}
        I_{21}=&d^2s\Re \int_{-T}^T\int_\Omega \pxx \varphi \pxx w\overline{w}\,dx\,dt\\ 
        =&-d^2s\int_{-T}^T \int_\Omega \pxx \varphi |\px w|^2\,dx\,dt -d^2 s\Re \int_{-T}^T\int_\Omega \pxxx \varphi \overline{w} \px w\,dx\,dt \\
        &+d^2s\Re \int_{-T}^T \pxx \varphi(0,\cdot) \overline{w}(0,\cdot)\px w(0,\cdot)dt.
    \end{align*}

    Since $|\px^3\varphi|\leq  C\lambda^3\theta$, we have 
    \begin{align*}
        I_{21}\geq & ds\int_{-T}^T\int_\Omega (\lambda^2 \theta |\psi'|^2+\lambda\theta\psi'')|\px w|^2\,dx\,dt -Cs^2\lambda^4 \int_{-T}^T \int_\Omega \theta |w|^2\,dx\,dt\\
        &-C\lambda^2\int_{-T}^T \int_\Omega \theta |\px w|^2\,dx\,dt -Cs^2\lambda^3 \int_{-T}^T \theta |w(o,\cdot)|^2\,dt -C\lambda \int_{-T}^T \theta |\px w(\cdot,0)|^2\,dt.
    \end{align*}

    On the other hand, after using integration by parts, the term $I_{22}$ can be written as:
    \begin{align*}
        I_{22}=&d^2s\int_{-T}^T\int_\Omega \px \varphi \px |\px w|^2\,dx\,dt\\
        =&d^2s\lambda \int_{-T}^T\int_\Omega (\lambda |\psi '|^2+\psi'')\theta |\px w|^2\,dx\,dt -d^2s\lambda \int_{-T}^T \psi'(\ell) \theta(\ell,\cdot)|\px w(\ell,\cdot)|^2\,dt\\
        &+d^2s\lambda \int_{-T}^T \psi'(0)\theta(0,\cdot)|\px w(0,\cdot)|^2\,dt.
    \end{align*}

    Besides, using $\Re(iz)=-\Im(z)$ for all $z\in \mathbb{C}$, we have 
    \begin{align*}
        I_{23}=&ds\Im \int_{-T}^T\int_\Omega \pt \varphi \overline{w}\pxx w\,dx\,dt\\
        =&-ds\Im \int_{-T}^T\int_\Omega \px\pt \varphi \overline{w}\px w\,dx\,dt +ds\Im \int_{-T}^T \pt \varphi(0,\cdot) \overline{w}(0,\cdot) \px w(0,\cdot)\,dt , 
    \end{align*}
    where we have used the fact that $w(L,t)=0$ for all $t\in (0,T)$. Now, using the estimates $|\pt \varphi|\leq C\lambda \theta^2$ and $|\px\pt \varphi|\lambda^2 \theta^2$ and Young's inequality, for all $\epsilon>0$, there exists $C(\epsilon)>0$ such that 
    \begin{align*}
        I_{23}\geq &-\int_{-T}^T \int_\Omega (\epsilon s\lambda \theta |\px w|^2 + C(\epsilon)s\lambda^3 \theta^3|w|^2)\,dx\,dt\\
        &-\int_{-T}^T (\epsilon s\lambda \theta(0,\cdot)|\px w(0,\cdot)|^2 + C(\epsilon)s\lambda \theta^3(0,\cdot)|w(0\cdot)|^2)\,dt
    \end{align*}

    Thus, choosing $\epsilon>0$ small enough and taking $\lambda_0,s_0>0$ sufficiently large, we deduce that 
    \begin{align}
        \label{P2}
        \begin{split} 
        &\sum_{k=1}^3 I_{2k}\\
        \geq &Cs\lambda^2 \int_{-T}^T\int_\Omega \theta |\px w|^2\,dx\,dt +Cs\lambda \int_{-T}^T \theta(0,\cdot)|\px w(0,\cdot)|^2\,dt -Cs\lambda \int_{-T}^T \theta(\ell,\cdot) |\px w(\ell,t)|^2\,dt\\
        &-Cs^2\lambda^3 \int_{-T}^T \theta(0,\cdot)|w(0,\cdot)|^2\,dt -Cs\lambda^4 \int_{-T}^T \int_\Omega \theta |w|^2\,dx\,dt.
        \end{split}
    \end{align}

    Now, let us estimate the terms $I_{3k}$, for $k=1,2,3$. Observe that $I_{31}$ can be written in the form 
    \begin{align*}
        I_{31}=-ds\Im \int_{-T}^T \int_\Omega \pxx \varphi \overline{w}\pt w\,dx\,dt.
    \end{align*}

    We point out that this term cannot be estimated directly. Indeed, this term will be eliminated when we summing up the $I_{3k}$, $k=1,2,3$. 

    Now, after integration by parts in space, the term $I_{32}$ can be written as follows:
    \begin{align*}
        I_{32}=&-2ds\Im  \int_{-T}^T\int_\Omega \px \varphi \pt w \px \overline{w}\,dx\,dt\\
        =&2ds \Im \int_{-T}^T\int_\Omega \pxx \varphi \pt w\overline{w}\,dx\,dt +2ds\Im \int_{-T}^T\int_\Omega \px \varphi \px \pt w\overline{w}\,dx\,dt\\
        &-2ds\Im \int_{-T}^T \px \varphi (0,\cdot)\pt w(0,\cdot)\overline{w}(0,\cdot)\,dt.
    \end{align*}

    Now, integrating by parts in time and using the fact that $w(\cdot,\pm T)=0$, we have
    \begin{align*}
        I_{32}=&2ds\Im \int_{-T}^T\int_\Omega \pxx \varphi \overline{w}\pt w\,dx\,dt -2ds\Im \int_{-T}^T\int_\Omega \pt\px \varphi \overline{w}\px w\,dx\,dt-I_{32}\\
        &-2ds\Im \int_{-T}^T\px \varphi(0,\cdot)\overline{w}(0,\cdot)\pt w(0,\cdot)\,dt.
    \end{align*}

    Then, by Young's inequality, for all $\epsilon>0$, there exists $C(\epsilon)>0$ such that 
    \begin{align*}
        I_{32}\geq &-\epsilon s\lambda \int_{-T}^T\int_\Omega \theta |\px w|^2\,dx\,dt -C(\epsilon)s\lambda \int_{-T}^T \int_\Omega \theta |w|^2\,dx\,dt +ds\Im \int_{-T}^T\int_\Omega \pxx \varphi \overline{w}\pt w\,dx\,dt\\
        &+ds\lambda \Im \int_{-T}^T \psi'(0)\overline{w}(0,\cdot)\pt w(0,\cdot)\,dt.
    \end{align*}

    Moreover, $I_{33}$ is given by 
    \begin{align*}
        I_{33}=-\frac{1}{2}s\Re\int_{-T}^T\int_\Omega \ptt \varphi |w|^2\,dx\,dt \geq -Cs\int_{-T}^T\int_\Omega \theta^3 |w|^2\,dx\,dt, 
    \end{align*}
    where we have used integration by parts in time and $w(\cdot,\pm T)=0$.

    Thus, considering these estimates, we obtain
    \begin{align}
        \label{P3}
        \begin{split} 
        \sum_{k=1}^3 I_{3k}\geq &-\epsilon s\lambda \int_{-t}^T\int_\Omega \theta |\px w|^2\,dx\,dt -C(\epsilon)s\lambda^3 \int_{-T}^T\int_\Omega \theta^3 |w|^2\,dx\,dt\\
        &+ds\lambda \int_{-T}^T \psi'(0)\overline{w}(0,\cdot)\pt w(0,\cdot)\,dt.
        \end{split} 
    \end{align}

    Now, adding inequalities \eqref{P1}, \eqref{P2} and \eqref{P3}, choosing $\epsilon>0$ small enough and taking $\lambda_0,s_0$ sufficiently large, we deduce that 
    \begin{align}
        \label{estimate:P1P2}
        \begin{split} 
        &\Re \int_{-T}^T\int_\Omega P_1w\overline{P_2w}\,dx\,dt\\
        \leq & C\int_{-T}^T\int_\Omega (s^3\lambda^4 \theta^3 |w|^2+s\lambda^2 \theta |\px w|^2 )\,dx\,dt +C\int_{-T}^T (s^3\lambda^3 \theta^3(0,\cdot)|w(0,\cdot)|^2 + s\lambda \theta (0,\cdot)|\px w(0,\cdot)|^2)\,dt\\
        &-Cs\lambda \int_{-T}^T \theta(\ell,\cdot)|\px w(\ell,\cdot)|^2\,dt +ds\lambda \Im \int_{-T}^T \psi'(0)\overline{w}(0,\cdot)\pt w(0,\cdot)\,dt.
        \end{split} 
    \end{align}

    \textit{Step 3:} In this step, we will compute the terms 
    \begin{align*}
        \Re \int_{-T}^T Q_1w \overline{Q_2 w}\,dt =J_1+J_2.
    \end{align*}

    Observe that 
    \begin{align*}
        J_1=-ds\lambda \Im \int_{-T}^T \psi'(0)\overline{w}(0,\cdot)\pt w(0,\cdot)\,dt. 
    \end{align*}

    Moreover, after integration by parts, the term $J_2$ can be estimated as
    \begin{align*}
        J_2=&s\Re\int_{-T}^T \pt \varphi(0,\cdot) \overline{w}(0,\cdot) \pt w(0,\cdot)\,dt=-\dfrac{1}{2}s\Re \int_{-T}^T \ptt \varphi(0,\cdot) |w(0,\cdot)|^2\,dt\\
        \geq &-Cs\lambda \int_{-T}^T\theta^2(0,\cdot) |w(0,\cdot)|^2\,dt.
    \end{align*}

    Then, we conclude that 
    \begin{align}
        \label{estimate:Q1Q2}
        \Re\int_{-T}^T Q_1w\overline{Q_2w}\,dt \geq -Cs\lambda \int_{-T}^T \theta^2 |w(0,\cdot)|^2\,dt -ds\Im \int_{-T}^T \psi'(0) \overline{w}(0,\cdot)\pt w(0,\cdot)\,dt.
    \end{align}

    \textit{Step 4:} Combining \eqref{estimate:P1P2:Q1:Q2:R}, \eqref{estimate:P1P2} and \eqref{estimate:Q1Q2} and taking $s_0$ and $\lambda_0$ large enough we obtain 
    \begin{align*}
        &\int_{-T}^T\int_\Omega (|P_1w|^2 + |P_2w|^2)\,dx\,dt +\int_{-T}^T (|Q_1w|^2+|Q_2w|^2)\,dt\\
        &+\int_{-T}^T\int_\Omega (s^3\lambda^4 \theta^3 |w|^2+s\lambda^2 \theta |\px w|^2)\,dx\,dt\\
        &+\int_{-T}^T (s^3\lambda^3 \theta^3(0,\cdot)|w(0,\cdot)|^2 + s\lambda \theta(0,\cdot)|\px w(0,\cdot)|^2)\,dt\\
        \leq & C\int_{-T}^T\int_\Omega |Pw|^2\,dx\,dt +C\int_{-T}^T|Qw|^2\,dt +Cs\lambda \int_{-T}^T \theta(\ell,\cdot)|\px w(\ell,\cdot)|^2\,dt,
    \end{align*}
    where we absorbed the term $R_\Gamma w$ by taking $s_0$ and $\lambda_0$ large enough. Finally, we come back to the original variables taking into account that 
    \begin{align*}
        e^{-2s\varphi}|\px v|^2\leq s^2\lambda^2 |w|^2+|\px w|^2,
    \end{align*}
    to obtain \eqref{Carleman:1D}. This ends the proof of Lemma \ref{lemma:Carleman:1D}.
\end{proof}

\bibliography{biblio02}
\bibliographystyle{plain}	
\end{document}